\documentclass[a4paper,11pt,reqno]{amsart}

\usepackage[english]{babel}
\usepackage[utf8]{inputenc}
\usepackage{setspace}

\usepackage{amsmath,nccmath,esdiff}
\usepackage{amsfonts,amssymb,latexsym,graphics,amsthm}
\usepackage[bookmarks=false]{hyperref}
\usepackage[dvips]{color}
\usepackage[dvips]{graphicx}
\usepackage{pstricks,pst-tree,rotating,stmaryrd}
\usepackage{url,multirow}

\renewenvironment{itemize}{\begin{list}{\labelitemi}{\leftmargin=1.5em}}{\end{list}}
\renewcommand{\labelitemi}{$\bullet$}

\DeclareMathOperator{\cro}{cr}
\DeclareMathOperator{\sch}{sc}
\renewcommand{\neg}{ {\rm neg} }

\DeclareMathOperator{\emp}{em}
\DeclareMathOperator{\cc}{cc}
\DeclareMathOperator{\fh}{fh}
\DeclareMathOperator{\w}{w}

\DeclareMathOperator{\dda}{dda}
\DeclareMathOperator{\va}{va}
\DeclareMathOperator{\pk}{pk}

\newcommand{\nodr}{ \psset{unit=2mm}\begin{pspicture}(-1,-0.6)(1.4,1.6)
\psdot(0,0)
\pscurve(0,0)(0.7,0.8)(1,1)\pscurve[linestyle=dotted,dotsep=0.2mm](1,1)(1.3,1.2)(1.6,1.3)
\pscurve(0,0)(0.7,-0.8)(1,-1)\pscurve[linestyle=dotted,dotsep=0.2mm](1,-1)(1.3,-1.2)(1.6,-1.3)
\end{pspicture} }
\newcommand{\nodh}{\psset{unit=2mm}\begin{pspicture}(-1,-0.6)(1.4,1.6)
 \psdot(0,0)
\pscurve(0,0)(0.7,0.8)(1,1)\pscurve[linestyle=dotted,dotsep=0.2mm](1,1)(1.3,1.2)(1.6,1.3)
\pscurve(0,0)(-0.7,0.8)(-1,1)\pscurve[linestyle=dotted,dotsep=0.2mm](-1,1)(-1.3,1.2)(-1.6,1.3)
\end{pspicture}}
\newcommand{\nodb}{\psset{unit=2mm}\begin{pspicture}(-1,-0.6)(1.4,1.6)
 \psdot(0,0)
\pscurve(0,0)(0.7,-0.8)(1,-1)\pscurve[linestyle=dotted,dotsep=0.2mm](1,-1)(1.3,-1.2)(1.6,-1.3)
\pscurve(0,0)(-0.7,-0.8)(-1,-1)\pscurve[linestyle=dotted,dotsep=0.2mm](-1,-1)(-1.3,-1.2)(-1.6,-1.3)
\end{pspicture}}
\newcommand{\nodg}{\psset{unit=2mm}\begin{pspicture}(-1,-0.7)(1.4,1.3)
 \psdot(0,0)
\pscurve(0,0)(-0.7,0.8)(-1,1)\pscurve[linestyle=dotted,dotsep=0.2mm](-1,1)(-1.3,1.2)(-1.6,1.3)
\pscurve(0,0)(-0.7,-0.8)(-1,-1)\pscurve[linestyle=dotted,dotsep=0.2mm](-1,-1)(-1.3,-1.2)(-1.6,-1.3)
\end{pspicture} }
\newcommand{\nodf}{ \psset{unit=2mm}\begin{pspicture}(-1,-0.6)(1.4,1.6)
 \psdot(0,0)\pscurve(0,0)(-0.3,0.6)(0,0.8)(0.3,0.6)(0,0) \rput(0,1.2){$-$}\end{pspicture}
}

\newcommand{\nodl}{ \psset{unit=2.5mm}\begin{pspicture}(-1,-0.6)(1.4,0.6)
 \psdot(0,0)\pscurve(0,0)(-0.3,0.6)(0,0.8)(0.3,0.6)(0,0) \end{pspicture}
}
\newcommand{\nodll}{ \psset{unit=2.5mm}\begin{pspicture}(-1,-0.6)(1.4,0.6)
 \psdot(0,0)\pscurve(0,0)(-0.3,-0.6)(0,-0.8)(0.3,-0.6)(0,0) \end{pspicture}
}

\title[Enumeration of snakes and cycle-alternating permutations]
  {Enumeration of snakes and \\ cycle-alternating permutations}
\date{\today}
\author{Matthieu Josuat-Vergès}
\dedicatory{To the memory of Vladimir Arnol'd}
\thanks{Supported by the French National Research Agency ANR, grant ANR08-JCJC-0011, and the Austrian Science Foundation FWF, START grant Y463}
\address{Fakultät für Mathematik, Universität Wien, 1090 Wien, Austria}
\email{Matthieu.Josuat-Verges@univie.ac.at}

\hypersetup{
    unicode=false,          
    pdftoolbar=true,        
    pdfmenubar=true,        
    pdffitwindow=false,     
    pdfstartview={FitH},    
    pdftitle={Enumeration of snakes and cycle-alternating permutations}, 
    pdfauthor={Matthieu Josuat-Vergès (Universität Wien)},
    pdfsubject={Enumerative and bijective combinatorics},
    pdfkeywords={enumeration, snakes, alternating permutations, Euler numbers, Springer numbers, continued fractions},
    pdfnewwindow=true,      
    colorlinks=true,       
    linkcolor=red,          
    citecolor=blue,        
    filecolor=magenta,      
    urlcolor=cyan           
}

\newtheorem{thm}{Theorem}[section]

\newtheorem{lem}[thm]{Lemma}
\newtheorem{prop}[thm]{Proposition}

\theoremstyle{definition}
\newtheorem{defn}[thm]{Definition}
\newtheorem{rem}[thm]{Remark}

\begin{document}
\newcommand{\bla}{
  \psset{unit=1mm} \begin{pspicture}(0,0) \rput(0,2){$\ddots$}   \end{pspicture}
}
\renewcommand{\qedsymbol}{
\psset{unit=1mm}\begin{pspicture}(1,3)\psframe[fillstyle=solid,fillcolor=black](0,0)(2,3)      \end{pspicture}
}

\begin{abstract}
Springer numbers are an analog of Euler numbers for the group of signed permutations. Arnol'd showed
that they count some objects called snakes, that generalize alternating permutations. Hoffman established
a link between Springer numbers, snakes, and some polynomials related with the successive derivatives
of trigonometric functions.

The goal of this article is to give further combinatorial properties of derivative polynomials, in terms of
snakes and other objects: cycle-alternating permutations, weighted Dyck or Motzkin paths, increasing
trees and forests. We obtain the generating functions, in terms of trigonometric functions for exponential
ones and in terms of $J$-fractions for ordinary ones. We also define natural $q$-analogs, make a link
with normal ordering problems and combinatorial theory of differential equations.
\end{abstract}

\maketitle

\section{Introduction}

It is well-known that the Euler numbers $E_n$ defined by
\begin{equation}
\sum_{n=0}^\infty E_n \frac{z^n}{n!} = \tan z + \sec z
\end{equation}
count alternating permutations in $\mathfrak{S}_n$, {\it i.e.} $\sigma$ such that
$\sigma_1>\sigma_2<\sigma_3>\dots \sigma_n$. The study of these, as well as other 
classes of permutations counted by $E_n$, is a vast topic in enumerative 
combinatorics, see the survey of Stanley~\cite{Sta}. By a construction due to
Springer~\cite{Spr}, there is an integer $K(W)$ defined for any Coxeter group $W$
such that $K(\mathfrak{S}_{n})=E_n$. As for the groups of signed permutations,
Arnol'd~\cite{Arn} introduced some particular kind of signed permutations called snakes,
counted by the numbers $K(\mathfrak{S}^B_n)$ and $K(\mathfrak{S}^D_n)$. Thus snakes
can be considered as a ``signed analog'' of alternating permutations. Algorithmically,
Arnol'd~\cite{Arn} gives a method to compute these integers with recurrences organized in
triangular arrays similar to the Seidel-Entriger triangle of Euler numbers $E_n$
(see for example \cite{GZ}).

In this article, we are mostly interested in the numbers $S_n=K(\mathfrak{S}^B_n)$, which were presviously
considered by Glaisher~\cite[§§ 109 and 119]{Gla} in another context. Springer~\cite{Spr} shows that they satisfy
\begin{equation}
\sum_{n=0}^\infty S_n \frac{z^n}{n!} = \frac{1}{\cos z - \sin z},
\end{equation}
and we call $S_n$ the $n$th Springer number (although there is in theory a Springer number associated with
each Coxeter group, this name without further specification usually refer to type $B$). Another link between snakes
and trigonometric functions has been established by Hoffman~\cite{Hof}, who studies polynomials associated
with the successive derivatives of $\tan$ and $\sec$, {\it i.e.} $P_n(t)$ and $Q_n(t)$ such that:
\begin{equation} \label{def_pq}
\diff[n]{}{x} \tan x = P_n(\tan x), \qquad 
\diff[n]{}{x} \sec x = Q_n(\tan x) \sec x.
\end{equation}
Among various other results, Hoffman proves that $ Q_n(1)=S_n$, and that $P_n(1)=2^nE_n$ counts some 
objects called $\beta$-snakes that are a superset of type $B$ snakes. A first question we can ask is
to find the meaning of the parameter $t$ in snakes.

The main goal of this article is to give combinatorial models of these derivative polynomials
$P_n$ and $Q_n$ (as well as another sequence $Q_n^{(a)}$ related the with the $a$th power of $\sec$),
in terms of several objects:
\begin{itemize}
 \item Snakes (see Definition~\ref{def_sna}). Besides type $B$ snakes of Arnol'd and $\beta$-snakes of Hoffman
         \cite{Hof}, we introduce another variant.
 \item Cycle-alternating permutations (see Definition~\ref{def_cyc}). These are essentially the image of snakes via
         Foata's fundamental transform \cite{lothaire}.
 \item Weighted Dyck prefixes and weighted Motzkin paths (see Section~\ref{sec:com}). These are two different
         generalizations of some weighted Dyck paths counted by $E_n$.
 \item Increasing trees and forests (see Section~\ref{sectree}). The number $E_{2n+1}$ counts increasing complete
         binary trees, and $E_{2n}$ ``almost complete'' ones, our trees and forests are a generalization of these.
\end{itemize}
Various exponential and ordinary generating functions are proved combinatorially using these objects. They
will present a phenomenon similar to the case of Euler numbers $E_n$:
they are given in terms of trigonometric functions for the exponential generating functions, and in terms
of continued fractions for the ordinary ones. For example, in Theorem~\ref{jfracqnab}, we prove
bijectively a continued fraction for the (formal) Laplace transform of $(\cos z - t \sin z)^{-a}$. While this result
was known analytically since long ago \cite{stieltjes}, we show here that it can be fully understood on the
combinatorial point of view.

\bigskip

This article is organized as follows. 
In Section~\ref{sec:der}, we give first properties of the derivative polynomials, such as recurrence relations, generating
functions, we also introduce a natural $q$-analog and link this with the normal ordering problem.
In Section~\ref{sec:com}, we give combinatorial models of the derivative polynomials
in terms of subset of signed permutations, and obtain the generating functions.
In Section~\ref{sectree}, we give combinatorial models of the derivative polynomials
in terms of increasing trees and forests in two different ways: via the combinatorial theory of differential equations
and via the normal ordering problem.

\section{Derivative polynomials}
\label{sec:der}

\subsection{Definitions}
Generalizing $Q_n(t)$ defined in \eqref{def_pq}, we consider $Q^{(a)}_n(t)$ such that
\begin{equation} \label{def_qa}
\diff[n]{}{x} \sec^ax = Q^{(a)}_n(\tan x) \sec^a x.
\end{equation}
Besides $Q_n$ which is the case $a=1$ in $Q^{(a)}_n$, of particular interest will be
the case $a=2$ denoted by $R_n=Q_n^{(2)}$.
Since $\sec^2x$ is the derivative of $\tan x$, it follows that there is the simple relation
$P_{n+1}(t)=(1+t^2)R_n(t)$.

By differentiation of \eqref{def_pq} and \eqref{def_qa}, one obtains the recurrence relations:
\begin{equation} \label{rec_pqa} \begin{split}
P_{n+1} & = (1+t^2)P_n', \qquad\qquad  Q^{(a)}_{n+1}= (1+t^2)\diff{}{t}Q^{(a)}_n + atQ_n, \\
\end{split}
\end{equation}
together with $P_0=t$ and $Q^{(a)}_0=1$. Let us give other elementary properties which are
partly taken from \cite{Hof}.

\begin{prop} \label{elem}
The polynomial $Q^{(a)}_n$ has the same parity as $n$, {\it i.e.}
$Q^{(a)}_n(-t)=(-1)^nQ^{(a)}_n(t)$, and $P_n$ has different parity.
We have:
\begin{equation}
P_n(1) = 2^n E_n, \qquad Q_n(1)=S_n, \qquad R_n(1)=2^nE_{n+1}.
\end{equation}
(Recall that $Q_n=Q_n^{(1)} $ and $R_n=Q^{(2)}_n$.) The generating functions are:
\begin{align}
\sum_{n=0}^\infty P_n(t) \frac{z^n}{n!} &= \frac{\sin z + t\cos z}{\cos z - t\sin z}, \label{genpq}
& \sum_{n=0}^\infty Q^{(a)}_n(t) \frac{z^n}{n!} = \frac1{(\cos z - t \sin z)^a}.
\end{align}
\end{prop}

\begin{proof}
The generating function of $\{Q^{(a)}_n\}_{n\geq0}$  is not present in the reference \cite{Hof}, but can be obtained
in the same way as in the particular case $a=1$ (which is in \cite{Hof}). By a Taylor expansion and a
trigonometric addition formula, we have:
\begin{equation}
\sum_{n=0}^\infty Q^{(a)}_n(\tan u)  \sec^a u \frac{z^n}{n!} = \sec^a(u+z) = \frac{ \sec^a u }{(\cos z - \tan u \sin z)^a},
\end{equation}
We can divide on both sides by $\sec^a u$, let $t=\tan u$, and the result follows.
\end{proof}

These exponential generating functions will be obtained combinatorially in Section
\ref{sec:com} using classes of signed permutations. It is interesting to note that $P(z,t)=\sum P_n z^n/n! $
is a Möbius transformation as a function of $t$, in such a way that $z\mapsto ( t\mapsto P(z,t) )$
is a group homomorphism from $\mathbb{R}$ to the elliptic Möbius transformations of $\mathbb{C}$
fixing $i$ and $-i$. This has an explanation through the fact that
\begin{equation}
 P(z,t) = \tan( \arctan t + z  ),
\end{equation}
and consequently, as observed in \cite{Hof}:
\begin{equation}
 P(z,P(z',t)) = P(z+z',t),
\end{equation}
which is the concrete way to say that $z\mapsto ( t\mapsto P(z,t) )$ is a group homomorphism.


\subsection{Operators and $q$-analogs}
Let $D$ and $U$ be operators acting on polynomials in the variable $t$ by 
\begin{equation} \label{def_du}
D(t^n) = [n]_qt^{n-1}, \qquad U(t^n) = t^{n+1},
\end{equation}
where $[n]_q=\frac{1-q^n}{1-q}$. The first one is known as the {\it $q$-derivative} or 
{\it Jackson derivative}, and an important relation is $DU-qUD=I$, where $I$ is the identity.

\begin{defn}
Our $q$-analogs of the polynomials $Q_n$ and $R_n$ are defined by:
\begin{equation} \label{def_qr}
Q_n(t,q) = (D+UDU)^n 1, \qquad R_n(t,q) = (D+DUU)^n1.
\end{equation}
\end{defn}

Of course, $1$ should be seen as $t^0$ since the operators act on polynomials in $t$.

\begin{prop}
We have 
\begin{equation}
Q_n(-t,q)=(-1)^nQ_n(t,q), \qquad R_n(-t,q)=(-1)^nR_n(t,q),  
\end{equation}
moreover $Q_n(t,q)$ and $R_n(t,q)$ are 
polynomials in $t$ and $q$ with nonnegative coefficients such that $Q_n(t,1)=Q_n(t)$ and $R_n(t,1)=R_n(t)$. 
\end{prop}

\begin{proof}
From $DU-qUD=I$, we can write 
\begin{equation} \label{relpq}
D+UDU = (I+qU^2)D + U, \qquad D+DUU = (I+q^2U^2)D + (1+q)U,
\end{equation}
then from \eqref{def_qr} and \eqref{relpq} it follows that
\begin{align}
Q_{n+1}(t,q) &= (1+qt^2)   D\big( Q_n(t,q) \big) + tQ_n(t,q),       \label{rec_Q}   \\
R_{n+1}(t,q) &= (1+q^2t^2) D\big( R_n(t,q) \big) + (1+q)tR_n(t,q),  \label{rec_R}
\end{align}
which generalize the recurrences for $Q_n$ and $R_n$ ($a=1$ and $a=2$ in \eqref{rec_pqa}). The
elementary properties given in the proposition can be proved recursively using \eqref{rec_Q} and \eqref{rec_R}.
\end{proof}

The fact that $Q_n(t,q)$ and $R_n(t,q)$ have simple forms in terms of $D$ and $U$ makes a link 
with {\it normal ordering} \cite{Bla}. Given an expression $f(D,U)$ of $D$ and $U$, the problem
is to find some coefficients $c_{i,j}$ such that we have the {\it normal form}:
\begin{equation} \label{normalorder}
   f(D,U) = \sum_{i,j\geq0} c_{i,j} U^iD^j.
\end{equation}
When $f(D,U)$ is polynomial, using the commutation relation $DU-qUD=I$ we can always find such coefficients
and only finitely of them are non-zero. For example, this is what we have done in \eqref{relpq}.
In Section~\ref{sectree}, we will use some general results on normal ordering to give combinatorial
models of $P_n$ and $Q_n$ (when $q=1$).

What is interesting about normal ordering is that having the coefficients $c_{i,j}$ give more insight on the quantity
$f(D,U)$. For example, consider the case of $Q_n(t,q)$ and let $f(D,U)=(D+UDU)^n$, then we have
\begin{equation}
  Q_n(t,q) =  (D+UDU)^n1 = \sum_{i,j\geq0} c_{i,j} U^iD^j1 =  \sum_{i\geq0} c_{i,0}t^i,
\end{equation}
{\it i.e.} we can directly obtain the coefficients of $Q_n(t,q)$ through this normal form.

\subsection{Continued fractions}
Let $D_1$ and $U_1$ be the matrices of operators $D$ and $U$ in the basis 
$\{t^i\}_{i\in\mathbb{N}}$. Moreover, let $W_1$ be the row vector $(t^i)_{i\in\mathbb{N}}$
and $V_1$ be the column vector $(\delta_{i0})_{i\in\mathbb{N}}$ where we use Kronecker's
delta. We have $W_1=(1,t,t^2,\dots)$, and:
\begin{equation}
  D_1=\left(\begin{mmatrix}
    0 & [1]_q  &   &  (0) \\
      & 0 &  [2]_q     \\
      &   & 0   &  \bla \\ 
   (0)   &   &  &  \bla \\
  \end{mmatrix}\right), \qquad
  U_1=\left(\begin{mmatrix}
    0 &  &  &  (0) \\
    1 & 0 &       \\
      & 1 & 0 &  \\
   (0)&   & \bla  & \bla  \\
  \end{mmatrix}\right), \qquad
  V_1 = \left(\begin{mmatrix}
     1 \\
     0 \\
     0 \\
    \vdots \\
  \end{mmatrix}\right).
\end{equation}
From the previous definitions, it follows that:
\begin{align}
&  Q_n(t,q) = W_1(D_1+U_1D_1U_1)^nV_1, \qquad R_n(t,q) = W_1(D_1+D_1U_1U_1)^nV_1, \label{rel1}  \\
&  D_1U_1-qU_1D_1=I, \qquad W_1U_1=tW_1, \qquad D_1V_1=0. \label{rel2}
\end{align}
The point of writing this is that from \eqref{rel1} we only need the relations in \eqref{rel2}
to calculate $Q_n(t,q)$ and $R_n(t,q)$. Even more, if $D_2$, $U_2$, $W_2$, and $V_2$ are a 
second set of matrices and vectors satisfying relations similar to \eqref{rel2}, we also have
$Q_n(t,q) = W_2(D_2+U_2D_2U_2)^nV_2$ and $R_n(t,q) = W_2(D_2+D_2U_2U_2)^nV_2$.
Indeed, when we have the relation \eqref{normalorder} for a given $f(D,U)$, the coefficients
$c_{i,j}$ are obtained only using the commutation relation, so the same identity also holds
with either $(D_1,U_1)$ or $(D_2,U_2)$. We can take
$W_2=(1,0,0,\dots)$, $V_2=V_1$, $D_2=D_1$, and:
\begin{equation}
  U_2=\left(\begin{mmatrix}
    t &  &  &  (0) \\
    1 & tq &       \\
      & 1 & tq^2 &  \\
   (0)&   & \bla  & \bla  \\
  \end{mmatrix}\right).
\end{equation}
Then we have indeed $D_2U_2-qU_2D_2=I$, $W_2U_2=tW_2$, and $D_2V_2=0$. What is nice 
about this second set of matrices and vectors is that it enables us to make a link 
with continued fractions.

\begin{defn} \label{def_jfra}
For any two sequences $\{b_h\}_{h\geq0}$ and $\{\lambda_h\}_{h\geq1}$, let 
$\mathfrak{J}(b_h,\lambda_h)$ denote
\begin{equation} 
\mathfrak{J} (b_h,\lambda_h) =  
 \cfrac{1}{1 - b_0z - \cfrac{\lambda_1 z^2}{1 - b_1z -  \cfrac{\lambda_2z^2}{\ddots}}}.
\end{equation}
We will always use $z$ as variable and $h$ as index of the two sequences so that there should
be no ambiguity in the notation.
\end{defn}

These are called $J$-fractions, or Jacobi continued fractions. They are linked 
with moments of formal orthogonal polynomials, but let us give the results for 
$Q_n(t,q)$ and $R_n(t,q)$ before more details about this.

\begin{prop} \label{frac_pq}
We have:
\begin{equation}
\sum_{n=0}^\infty Q_n(t,q) z^n = \mathfrak{J}(b_h^Q, \lambda_h^Q), \qquad
\sum_{n=0}^\infty R_n(t,q) z^n = \mathfrak{J}(b_h^R, \lambda_h^R),
\end{equation}
where:
\begin{align}
b_h^Q &= tq^h([h]_q+[h+1]_q) , \qquad  &  b_h^R &= tq^h(1+q)[h+1]_q,  \\
\lambda_h^Q &= (1+t^2q^{2h-1}) [h]_q^2, \quad & \lambda_h^R &= (1+t^2q^{2h})[h]_q[h+1]_q.
\end{align}
\end{prop}

\begin{proof}
First, observe that for any matrix $M=(m_{i,j})_{i,j\in\mathbb{N}}$, the product $W_2M^nV_2$
is the upper-left coefficient $(M^n)_{0,0}$ of $M^n$. We can obtain this coefficient 
$W_2M^nV_2$ the following way:
\begin{equation} \label{expa}
 W_2M^nV_2 = \sum_{i_1,\dots,i_{n-1} \geq 0} m_{0,i_1}m_{i_1,i_2}\dots m_{i_{n-2},i_{n-1}} 
m_{i_{n-1},0}.
\end{equation}
When the matrix is tridiagonal, we can restrict the sum to indices such that $|i_j-i_{j+1}|\leq1$, so that 
these indices are the successive heights in a Motzkin path.
Then \eqref{expa} shows that $W_2M^nV_2$ is the generating function of Motzkin paths of length $n$
with some weights given by the coefficients of $M$, and using a classical argument \cite{Fla} it 
follows that $\sum_{n=0}^\infty (W_2M^nV_2)z^n = \mathfrak{J}(m_{h,h},m_{h-1,h}m_{h,h-1})$.

In the present case, it suffices to check that $D_2+U_2D_2U_2$ and $D_2+D_2U_2U_2$ are 
tridiagonal and calculate explicitly their coefficients to obtain the result. See also \eqref{qnatqrec}
and \eqref{qnatqbhlh} below for a more general result.
\end{proof}

In particular, it follows from the previous proposition that we have:
\begin{equation}
 Q_{2n}(0,q) = E_{2n}(q), \quad R_{2n}(0,q)=E_{2n+1}(q),
\end{equation}
where $E_{2n}(q)$ and $E_{2n+1}(q)$ are respectively the $q$-secant and $q$-tangent
numbers defined by Han, Randrianarivony, Zeng~\cite{HZR} using continued fractions.

\begin{rem}
From the fact that $Q_n(1)=S_n$ and $R_n(1)=2^{n}E_{n+1}$, the previous proposition implies
\begin{equation}
\sum_{n=0}^\infty S_n z^n = \mathfrak{J}(2h+1,2h^2), \qquad 
\sum_{n=0}^\infty E_{n+1} z^n = \mathfrak{J}(h+1,\tfrac{h(h+1)}2).
\end{equation}
We found only one reference mentioning the latter continued fraction, Sloane's OEIS \cite{OEIS},
but there is little doubt it can be proved by classical methods. For example, using a theorem
of Stieltjes and Rogers \cite[Theorem 5.2.10]{GJ}, the continued fraction can probably be obtained through an addition
formula satisfied by the exponential generating function of $\{E_{n+1}\}_{n\geq0}$, this function being the derivative
of $\tan z + \sec z$, explicitly $(1-\sin z)^{-1}$. Combinatorially, the result can be proved using {\it André trees}
\cite[Section 5]{FS2} and the bijection of Françon and Viennot \cite[Chapter 5]{GJ}. More generally, this shows that if we add a
parameter $x$, then $\mathfrak{J}(h+1,x\tfrac{h(h+1)}2)$ is the generating function of the {\it André polynomials} 
defined by Foata and Schützenberger~\cite{FS2}.
\end{rem}

An important property of $J$-fractions is the link with moments of (formal)
orthogonal polynomials, and we refer to \cite{Fla,Vie} for the relevant combinatorial 
facts. We consider here the {\it continuous dual $q$-Hahn polynomials} $p_n(x;a,b,c|q)$, or 
$p_n(x)$ for short. In terms of the Askey-Wilson polynomials which depend on one other
parameter $d$, $p_n(x)$ is just the specialization $d=0$ (see \cite{KoSw} for the
definitions of these classical sequences, but our notations differ, in particular because
of a rescaling $x\to x/2$). We have the three-term recurrence relation
\begin{equation} \label{rec_cdqh}
  xp_n(x) = p_{n+1}(x) + \left(a+\tfrac 1a - A_n - C_n\right)p_n(x) + A_{n-1}C_np_{n-1}(x),
\end{equation}
together with $p_{-1}(x)=0$ and $p_0(x)=1$, where
\begin{equation} \label{defAnCn}
 A_n =  \tfrac 1a (1-abq^n)(1-acq^n) \quad \hbox{ and } \quad C_n = a (1-q^n)(1-bcq^{n-1}).
\end{equation}
The $p_n(x)$ are orthogonal with respect to the scalar product $(f,g)\mapsto L(fg)$ where
$L$ is the linear form such that 
\begin{equation} \label{mu_frac}
\sum_{n=0}^\infty L(x^n) z^n = \mathfrak{J}\left(a+\tfrac1a-A_h-C_h,A_{h-1}C_h\right).
\end{equation}
The quantity $L(x^n)$ is called the $n$th {\it moment} of the orthogonal sequence $\{p_n(x)\}_{n\geq0}$.
Let $\mu_n(a,b,c)$ denote this $n$th moment of the continuous dual $q$-Hahn polynomials $p_n(x)$.
An elementary calculation shows the following:

%

\begin{prop} We have
\begin{equation} \label{muQR}
Q_n(t,q) = \frac{ \mu_n(i\sqrt q,-i\sqrt q, t)}{(1-q)^n},  \qquad 
R_n(t,q) = \frac{ \mu_n(iq,-iq, t) }{(1-q)^n}.
\end{equation}
\end{prop}

\begin{proof}
We just have to identify the ordinary generating functions of both sides in each identity.
This is possible because we know in each case the explicit form of the $J$-fraction expansion,
from \eqref{mu_frac} and Proposition~\ref{frac_pq}. See also \eqref{qnatqrec}
and \eqref{qnatqmu} below for a more general result.
\end{proof}

This very simple link between our polynomials and the moments $\mu_n(a,b,c)$ is one of the
properties indicating that $Q_n(t,q)$ and $R_n(t,q)$ are interesting $q$-analogs
of $Q_n(t)$ and $R_n(t)$. Note that in view of \eqref{rec_Q} and \eqref{rec_R}, it is tempting
to define a $q$-analog $Q_n^{(a)}(t,q)$ of $Q_n^{(a)}(t)$ by the recurrence:
\begin{equation} \label{qnatqrec}
  Q_{n+1}^{(a)}(t,q) = (1+q^at^2)D\Big( Q_n^{(a)}(t,q) \Big) + [a]_qtQ_n^{(a)}(t,q),
\end{equation}
together with $Q_0^{(a)}(t,q)=1$.
Though we will not study these apart the particular cases $a=1$ and $a=2$, it is worth mentioning
that a natural generalization of the previous proposition holds, more precisely we have:
\begin{equation} \label{qnatqmu}
  Q_n^{(a)}(t,q) = \frac{ \mu_n(iq^{a/2},-iq^{a/2}, t) }{(1-q)^n},
\end{equation}
which can be proved by calculating explicitly the matrix $(I+q^aU_2^2)D_2+[a]_qU_2$. The entries in this
matrix show that the generating function of $\{Q_n^{(a)}(t,q)\}$ has a $J$-fraction expansion with coefficients:
\begin{equation} \label{qnatqbhlh}
 b_h = tq^h\big( [h-1+a]_q + q^{a-1}[h+1]_q\big), \qquad \lambda_h = (1+t^2q^{2h-2+a})[h]_q[h-1+a]_q.
\end{equation}
Indeed we have the same coefficients if we replace $(a,b,c)$ with $(iq^{a/2},-iq^{a/2},t)$ in \eqref{defAnCn} and
\eqref{mu_frac}.

\section{Combinatorial models via signed permutations}
\label{sec:com}

\begin{defn} 
We will denote $[n]=\{1\dots n\}$ and $\llbracket n \rrbracket=\{{-n} \dots {-1}\}\cup\{1 \dots n\}$. A {\it signed permutation}
is a permutation $\pi$ of $\llbracket n \rrbracket$ such that $\pi(-i)=-\pi(i)$ for any $i\in \llbracket n \rrbracket $. It 
will be denoted $\pi=\pi_1\dots\pi_n$ where $\pi_i=\pi(i)$, and $\mathfrak{S}_n^\pm$ is the set of all such $\pi$.
We will also use the cycle notation, indicated by parenthesis: for example $\pi=3,-1,2,4$ is $(1,3,2,-1,-3,-2)(4)(-4)$ in 
cycle notation.
\end{defn}

We consider here two classes of signed permutations: snakes and cycle-alternating permutations. These will
give combinatorial models of the derivative polynomials. As for the $q$-analog, going through other objects (weighted
Dyck prefixes) we can obtain a statistic on cycle-alternating permutations counted by the parameter $q$.

\subsection{Snakes} There are several types of snakes to be distinguished, so let us first give the definition:

\begin{defn} \label{def_sna}
A signed permutation $\pi=\pi_1\dots\pi_n$ is a {\it snake} if
$\pi_1>\pi_2<\pi_3>\dots\pi_n$. We denote by $\mathcal{S}_n\subset\mathfrak{S}_n^\pm$ 
the set of snakes of size $n$. Let $\mathcal{S}^0_n\subset\mathcal{S}_n$ be the 
subset of $\pi$ satisfying $\pi_1>0$, and $\mathcal{S}^{00}_n\subset\mathcal{S}^0_n$ 
be the subset of $\pi$ satisfying $\pi_1>0$ and $(-1)^n\pi_n<0$.
\end{defn}

The objects introduced by Arnol'd~\cite{Arn} are the ones in $\mathcal{S}^0_n$, and
are usually called snakes of type $B$. He shows in particular that the cardinal of 
$\mathcal{S}^0_n$ is the $n$th Springer number $S_n$. The objects in $\mathcal{S}_n$
are called $\beta$-snakes by Hoffman~\cite{Hof}, and he showed that the cardinal of
this set is $P_n(1)=2^nE_n$.

It is practical to take a convention for the values $\pi_0$ and $\pi_{n+1}$.
If $\pi\in\mathcal{S}_n$ (respectively, $\pi\in\mathcal{S}^0_n$, $\pi\in\mathcal{S}^{00}_n$) 
we define $\pi_0=-(n+1)$ and $\pi_{n+1}=(-1)^n(n+1)$ (respectively, 
$\pi_0=0$ and $\pi_{n+1}=(-1)^n(n+1)$, $\pi_0=\pi_{n+1}=0$). This is not
consistent with the inclusions $\mathcal{S}^{00}_n\subset\mathcal{S}^0_n\subset\mathcal{S}_n$,
so we can think of these sets to be disjoint to avoid ambiguity in the sequel. Then an 
element of any of these three sets is a $\pi\in\mathfrak{S}^\pm_n$ such that 
$\pi_0<\pi_1>\pi_2<\pi_3>\dots\pi_{n+1}$. For example in the case of $\mathcal{S}^{00}_n$,
the condition $(-1)^n\pi_n<0$ means that if $n$ is odd, $\pi_n>\pi_{n+1}=0$, and if $n$ is 
even, $\pi_n<\pi_{n+1}=0$. For clarity, we will often write $\pi=(\pi_0),\pi_1,\dots,\pi_n,(\pi_{n+1})$.

\begin{defn}
Let $\pi$ be a snake in one of the sets $\mathcal{S}_n$, $\mathcal{S}^0_n$, or
$\mathcal{S}^{00}_n$, and $\pi_0$, $\pi_{n+1}$ as above. We define a statistic
$\sch(\pi)$ as the number of sign changes through the values $\pi_0\dots\pi_{n+1}$,
{\it i.e.} $\sch(\pi)=\#\{\;i\;|\;0\leq i\leq n, \; \pi_i\pi_{i+1}<0\;\}$.
\end{defn}

Using this statistic, we have combinatorial models of the derivative polynomials.

\begin{thm} \label{comb_snak}
For any $n\geq0$, we have:
\begin{equation}
P_n(t) = \sum_{\pi\in\mathcal{S}_n} t^{\sch(\pi)}, \qquad
Q_n(t) = \sum_{\pi\in\mathcal{S}^0_n} t^{\sch(\pi)}, \qquad
R_n(t) = \sum_{\pi\in\mathcal{S}^{00}_{n+1}} t^{\sch(\pi)}.
\end{equation}
\end{thm}

\begin{proof}
We check that the right-hand sides satisfy the recurrences in \eqref{rec_pqa} (case $a=1$ or $a=2$),
but we only detail the case of $R_n$ (case $a=2$), the other ones being similar. We have $R_0=1$,
and this corresponds to the snake $(0),1,(0)$. Also $R_1=2t$, and this corresponds to
$(0),1,-2,(0)$ and $(0),2,-1,(0)$. Suppose the result is proved for $R_{n-1}$, then to prove 
it for $R_n$ we distinguish three kinds of elements in $\mathcal{S}^{00}_{n+1}$
(the convention is $\pi_0=\pi_{n+2}=0$).
We will denote by $\pi_i^-$ the number $\pi_i-1$ if $\pi_i>0$ and $\pi_i+1$ if $\pi_i<0$.
\begin{itemize}
\item First, suppose that $|\pi_{n+1}|=1$, hence $\pi_{n+1}=(-1)^{n}$.
Let $\pi'=\pi_1^- \dots \pi_{n}^-$. Depending on the parity of $n$, we have either 
$\pi_{n}<\pi_{n+1}=1$ or $\pi_{n}>\pi_{n+1}=-1$, and it follows on one hand that 
$(-1)^n\pi_n<0$, on the other hand that $\pi_n\pi_{n+1}<0$.
Thus $\pi'\in\mathcal{S}^{00}_n$, and $\sch(\pi')=\sch(\pi)-1$. The map $\pi\mapsto\pi'$
is bijective, and with the recurrence hypothesis, it comes that the set of 
$\pi\in\mathcal{S}^{00}_{n+1}$ with $|\pi_{n+1}|=1$ is counted by $tR_{n-1}$.

\item Secondly, suppose that $\pi_1=1$, and let $\pi'=-\pi_2^-,\dots,{-\pi_{n+1}^-}$.
From $\pi_1>\pi_2$, we obtain $\pi_2<0$, hence $-\pi_2^->0$. It follows that 
$\pi'\in\mathcal{S}^{00}_n$. Moreover, since $\pi_1\pi_2<0$ we obtain 
$\sch(\pi')=\sch(\pi)-1$. With the recurrence hypothesis, it comes that the set of 
$\pi\in\mathcal{S}^{00}_{n+1}$ with $\pi_1=1$ is also counted by $tR_{n-1}$.

\item Thirdly, suppose that there is $j\in\{2\dots n\}$ such that $|\pi_j|=1$. We have either
$\pi_{j-1}>\pi_j<\pi_{j+1}$ or $\pi_{j-1}<\pi_j>\pi_{j+1}$, and it follows that 
$\pi_{j-1}$ and $\pi_{j+1}$ have the same sign. We will obtain the term $R_{n-1}'$ in
the subcase where $\pi_j$ have also the same sign as $\pi_{j-1}$ and $\pi_{j+1}$, and the
term $t^2R_{n-1}'$ otherwise. Let us prove the first subcase, the second will follow since
it suffices to consider snakes of the first subcase where $\pi_j$ is replaced with $-\pi_j$.

Let $\pi'=\pi_1^-,\dots,\pi_{j-1}^-,-\pi_{j+1}^-,\dots,{-\pi_{n+1}^-}$. Then we can check that
$\pi\mapsto(\pi',j)$ is a bijection between $\pi$ of the first subcase, and couples
$(\pi',j)$ where $\pi'\in\mathcal{S}^0_n$ and $j\in\{2\dots n\}$ is such that $\pi'_{j-1}\pi'_{j}<0$.
At the level of generating function, choosing a sign change in $\pi'$ is done by differentiation with
respect to $t$, and this explains why $\pi$ of the first subcase are counted by $R'_{n-1}$.
\end{itemize}
Adding the above terms, we obtain $(1+t^2)R'_{n-1}+2tR_{n-1} = R_n$. This completes the recurrence.
The difference in the case of $Q_n$ (respectively, $P_n$) is roughly that only the second and third points
(respectively, only the third point) are to be considered.
\end{proof}

The exponential generating function of $\{Q_n(t)\}_{n\geq0}$ can be directly obtained from 
snakes in the form 
\begin{equation} \label{expQ}
 \sum_{n=0}^\infty Q_n(t) \frac{z^n}{n!}  = \frac{\sec z}{1-t\tan z}.
\end{equation}
To prove this, let $\pi\in\mathcal{S}^0_n$, and consider its unique factorization $\pi=f_1\dots f_k$, where each
$f_k$ is a maximal non-empty factor such that all its entries have the same sign. Each of these factor has odd
length, except possibly the last one. So it is convenient to think that if the last factor $f_k$ has odd size, there
is also an empty factor $f_{k+1}$ at the end, and consider the new factorization $\pi=f_1\dots f_j$ (with $j=k$ or
$j=k+1$). Note that $\sch(\pi)=j-1$. The factorization proves \eqref{expQ}, because $\pi$ is built by assembling
alternating permutations of odd size $f_1\dots f_{j-1}$, and an alternating permutation of even size $f_j$.

Using the snakes, we can also directly obtain that
\begin{equation} \label{square}
\sum_{n=0}^\infty R_n(t) \frac{z^n}{n!} =
\left( \sum_{n=0}^\infty Q_n(t) \frac{z^n}{n!} \right)^2.
\end{equation}
Indeed, let $\pi=(0),\pi_1\dots\pi_{n+1},(0)$ be among snakes counted by $R_n(t)$, {\it i.e.} $\pi\in\mathcal{S}^{00}_{n+1}$.
There is $j\in\{1\dots n+1\}$ such that $|\pi_j|=n+1$. We define $\pi' = (0),\pi_1\dots\pi_{j-1},(\pi_j)$
and $\pi'' = (0), (-1)^{n}\pi_{n+1},\dots,(-1)^{n}\pi_{j+1}, ((-1)^{n}\pi_j)$.
After some relabelling, these define two snakes in $\mathcal{S}^0_{j-1}$ and $\mathcal{S}^0_{n-j+1}$.  
We have $\sch(\pi)=\sch(\pi')+\sch(\pi'')$, and $\pi$ is built by assembling $\pi'$ and $\pi''$, so this
proves \eqref{square}.

Another interesting question concerning snakes is the following. For a given permutation 
$\sigma=\sigma_1\dots\sigma_n\in\mathfrak{S}_n$, the problem is to choose signs
$\epsilon=\epsilon_1\dots\epsilon_n\in\{\pm1\}^n$ such that
$\pi=\epsilon_1\sigma_1,\dots,\epsilon_n\sigma_n$ is a snake. Arnol'd~\cite{Arn} described
the possible choices such that $\pi\in\mathcal{S}^0_n$ in terms of ascent and descent
sequences in the permutation. Let $i\in\{2\dots n-1\}$, then a case-by-case argument 
from \cite{Arn} show that:
\begin{equation} \label{signs}
\parbox{13cm}{
\begin{itemize}
\item if $\sigma_{i-1}<\sigma_i<\sigma_{i+1}$, then $\epsilon_{i}\neq\epsilon_{i+1}$,
\item if $\sigma_{i-1}>\sigma_i>\sigma_{i+1}$, then $\epsilon_{i-1}\neq\epsilon_{i}$,
\item if $\sigma_{i-1}>\sigma_i<\sigma_{i+1}$, then $\epsilon_{i-1}=\epsilon_{i+1}$.
\end{itemize}
}
\end{equation}
Following this, we can answer the problem of building a snake from a permutation in terms of some permutations statistics. 
We can apply this argument to the three sets $\mathcal{S}_n$, $\mathcal{S}^0_n$,  and $\mathcal{S}^{00}_n$, essentially 
by varying the conventions on $\sigma(0)$ and $\sigma(n+1)$.

\begin{defn}
Let $\sigma=\sigma_1\dots\sigma_n\in\mathfrak{S}_n$, we make the convention that 
$\sigma_0=\sigma_{n+1}=n+1$. We also need to consider other conventions, so let 
$\mathfrak{S}^0_n$ and $\mathfrak{S}^{00}_n$ be ``copies'' of the set $\mathfrak{S}_n$, where:
\begin{itemize}
  \item If $\sigma\in\mathfrak{S}^{0}_n$, we set $\sigma_0=0$ and $\sigma_{n+1}=n+1$.
  \item If $\sigma\in\mathfrak{S}^{00}_n$, we set $\sigma_0=\sigma_{n+1}=0$.
\end{itemize}
Let $\sigma$ be in one of the sets $\mathfrak{S}_n$, $\mathfrak{S}^0_n$, or $\mathfrak{S}^{00}_n$.
An integer $i\in[n]$ is 
a {\it valley} of $\sigma$ if $\sigma_{i-1}>\sigma_i<\sigma_{i+1}$, 
a {\it peak} if $\sigma_{i-1}<\sigma_i>\sigma_{i+1}$, 
a {\it double descent} if $\sigma_{i-1}>\sigma_i>\sigma_{i+1}$, and 
a {\it double ascent} if $\sigma_{i-1}<\sigma_i<\sigma_{i+1}$.
Let $\va(\sigma)$ denote the number of valleys in $\sigma$,
    $\pk(\sigma)$ denote the number of peaks in $\sigma$,  
and $\dda(\sigma)$ denote the number of double descents and double ascents in $\sigma$.
\end{defn}

For example, we see $1,3,2\in\mathfrak{S}^0_3$ as $(0),1,3,2,(4)$. Then $3$ is a valley.
On the other hand, we see $1,3,2\in\mathfrak{S}^{00}_3$ as $(0),1,3,2,(0)$. Then $3$ is a
double descent. 

\begin{prop} \label{propval}
When $n\geq1$, we have:
\begin{align}
P_n(t) &= \sum_{\sigma\in\mathfrak{S}_n} t^{\dda(\sigma)}(1+t^2)^{\va(\sigma)}
        = \sum_{\sigma\in\mathfrak{S}^{00}_n} t^{\dda(\sigma)}(1+t^2)^{\pk(\sigma)},  \label{statP}
\\
Q_n(t) &= \sum_{\sigma\in\mathfrak{S}^0_n} t^{\dda(\sigma)}(1+t^2)^{\va(\sigma)}, \qquad
R_n(t)  = \sum_{\sigma\in\mathfrak{S}^{00}_{n+1}} t^{\dda(\sigma)}(1+t^2)^{\va(\sigma)}.
\end{align}
The last two identities are also true when $n=0$.
\end{prop}

\begin{proof}
This follows from Theorem~\ref{comb_snak} and \eqref{signs}, which now can be extended to the cases $i=1$ or
$i=n$ using the conventions on $\sigma_0$ and $\sigma_{n+1}$ (note that we also need to consider $\epsilon_0$ and 
$\epsilon_{n+1}$). Let us first prove the case of $P_n$, $\sigma\in\mathcal{S}_n$ and $\pi\in\mathfrak{S}_n$. We
have $\epsilon_0=-1$ since we require that $\pi_0=-(n+1)$. The first two rules in \eqref{signs} show that to each
double ascent or double descent we can associate a sign change in $\pi$, and that the only remaining choices to
be done concern the valleys. At a valley $\sigma_{i-1}>\sigma_i<\sigma_{i+1}$, we can have either 
$\epsilon_{i-1}=\epsilon_i=\epsilon_{i+1}$, or $\epsilon_{i-1}=-\epsilon_i=\epsilon_{i+1}$, {\it i.e.} $0$ or $2$ sign changes.
This proves the first equality in \eqref{statP}, and the second follows by taking the complement permutation. The cases
of $Q_n$ and $R_n$ are proved similarly, except that we don't consider $\epsilon_0$ and we have $\epsilon_1=1$.
\end{proof}

An important property of the statistics $\va(\sigma)$ and $\dda(\sigma)$ is that they can be followed through the bijection
of Françon and Viennot (see for example the book \cite{GJ}). Omitting details, this gives combinatorial proof of:
\begin{align}
\sum_{n=0}^\infty Q_n(t) z^n &= \mathfrak{J}\big(\; (2h+1)t , \; (1+t^2)h^2 \;\big),  \label{fvfrac1}   \\
\sum_{n=0}^\infty R_n(t) z^n &= \mathfrak{J}\big(\; (2h+2)t , \; (1+t^2)h(h+1) \;\big), \label{fvfrac2}
\end{align}
which is the particular case $q=1$ of Proposition~\ref{frac_pq}. Unfortunately, it
seems difficult to follow the parameter $q$ through this bijection, and obtain
nice statistics on $\mathfrak{S}^{0}_{n}$ and $\mathfrak{S}^{00}_{n+1}$ corresponding
to $Q_n(t,q)$ and $R_n(t,q)$.

\subsection{Cycle-alternating permutations} From the combinatorial interpretation 
in terms of snakes, we can derive other ones using simple bijections. We need some general definitions
concerning signed permutations. Let $\pi\in\mathfrak{S}^{\pm}_n$.

\begin{defn} \label{def_cyc}
Let $\neg(\pi)$ be the number of $i\in[n]$ 
such that $\pi_i<0$. For any orbit $O$ of $\pi$, let $-O=\{-x\;|\;x\in O\}$. A {\it cycle} of $\pi$ is an unordered
pair of orbits $\{O_1,O_2\}$ such that $-O_1=O_2$. It is called a {\it one-orbit cycle} if $O_1=O_2$ and
{\it two-orbit cycle} otherwise. (This is known to be a natural notion of cycle since there is a decomposition
of any signed permutation into a product of disjoint cycles.) For any signed permutation $\pi$, its {\it arch
diagram} is defined as follows: draw on the horizontal axis $2n$ nodes labelled by the integers in
$\llbracket n \rrbracket$ in increasing order, then for any $i$ draw an arch from $i$ to $\pi(i)$ such that
the arch is above the horizontal axis if $i\leq\pi(i)$ and below the axis otherwise.
\end{defn}

See Figure~\ref{arrownot} for an example of an arch diagram. Since the labels from $-n$ to $n$ are in increasing
order we do not need to explicitly write them. See also Figure~\ref{conjug} further for another example.
We did not specify what happens with the fixed points in the arch diagram. Although this will not be really
important in the sequel, we can choose to put a loop \nodl at each positive fixed point and a loop \nodll at
each negative fixed point. This way, the arch diagram of a signed permutation is always centrally symmetric.

\begin{figure}[h!tp] \psset{unit=5mm}
\begin{pspicture}(-7,-1.7)(7,1.7)
\psdots(-1,0)(-2,0)(-3,0)(-4,0)(-5,0)(-6,0)(-7,0)
\psdots(1,0)(2,0)(3,0)(4,0)(5,0)(6,0)(7,0)
\pscurve(1,0)(1.5,0.4)(2,0) \pscurve(-1,0)(-1.5,-0.4)(-2,0)
\pscurve(2,0)(-1,-1.7)(-4,0) \pscurve(-2,0)(1,1.7)(4,0)
\pscurve(3,0)(4.5,1)(6,0)   \pscurve(-3,0)(-4.5,-1)(-6,0)
\pscurve(4,0)(2.5,-0.8)(1,0) \pscurve(-4,0)(-2.5,0.8)(-1,0)
\pscurve(5,0)(1,-1.9)(-3,0)  \pscurve(-5,0)(-1,1.9)(3,0)
\pscurve(6,0)(6.5,0.5)(7,0)  \pscurve(-6,0)(-6.5,-0.5)(-7,0)
\pscurve(7,0)(6,-0.8)(5,0)  \pscurve(-7,0)(-6,0.8)(-5,0)
\end{pspicture}
\caption{The arch diagram of $\pi=2,-4,6,1,-3,7,5$ \label{arrownot}}
\end{figure}
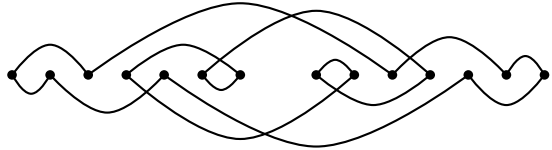

\begin{defn}
Let $\pi\in\mathfrak{S}_n^\pm$. An integer $i\in\llbracket n \rrbracket$ is a 
{\it cycle peak} of $\pi$ if $\pi^{-1}(i)<i>\pi(i)$, and it is a
{\it cycle valley} of $\pi$ if $\pi^{-1}(i)>i<\pi(i)$.
The signed permutation $\pi$ is {\it cycle-alternating} if every 
$i\in \llbracket n \rrbracket $
is either a cycle peak or a cycle valley. Let $\mathcal{C}_n\subset\mathfrak{S}^\pm_n$ denote the subset of 
cycle-alternating signed permutations, and $\mathcal{C}^\circ_n\subset\mathcal{C}_n$ be the subset of 
$\pi\in\mathcal{C}_n$ with only one cycle.
\end{defn}

\begin{lem} 
Let $\pi\in\mathcal{C}^\circ_n$. Then $n$ is even (respectively, odd) if and only if 
$\pi$ has a two-orbit cycle (repectively, a one-orbit cycle).
\end{lem}

\begin{proof}
In a cycle-alternating permutation, each orbit has even cardinal because there is an
alternance between cycle peaks and cycle valleys, so if there are two opposite orbits 
of size $n$, necessarily $n$ is even.

In the cycle-alternating permutation $\pi$, $-n$ is a cycle valley and $n$ is a cycle peak.
If there is one orbit, it can be written 
$\pi=(n,i_1,\dots,i_{n-1},-n,-i_1,\dots,-i_{n-1})$ in cycle notation.
Because of the alternance of cycles peaks and cycle valleys, it follows that $n$ is odd.
\end{proof}

\begin{thm}  \label{comb_cycl}
We have :
\begin{align}
P_n(t) = \sum_{\substack{\pi\in\mathcal{C}^\circ_{n+1} }} t^{\neg(\pi)},
\quad
Q_n(t) = \sum_{\pi\in\mathcal{C}_n} t^{\neg(\pi)},
\quad
R_n(t) = \sum_{\substack{\pi\in\mathcal{C}^\circ_{n+2} \\ \pi_1>0}} 
t^{\neg(\sigma)}.
\end{align}
\end{thm}

\begin{proof}
There are simple bijections between snakes and cycle-alternating permutations to prove 
this from Theorem~\ref{comb_snak}.

Let us begin with the case of $P_n$, so let $\pi\in\mathcal{S}_n$. Suppose first that $n$ is odd, 
hence $\pi_0=\pi_{n+1}=-(n+1)$. We can read $\pi$ in cycle notation and consider the two-orbit
cycle $\pi'=(\pi_0,\pi_1,\dots,\pi_n)(-\pi_0,-\pi_1,\dots,-\pi_n)$.
Then $\pi'\in\mathcal{C}^\circ_{n+1}$. When $n$ is even, we have $-\pi_0=\pi_{n+1}=n+1$, 
and the bijection is defined by taking $\pi'=(\pi_0,\pi_1,\dots,\pi_n,-\pi_0,-\pi_1,\dots,-\pi_n)$.
The map $\pi\mapsto\pi'$ is clearly invertible, and using the previous lemma, we have a bijection 
between $\mathcal{S}_n$ and $\mathcal{C}^\circ_{n+1}$.
Let $i\in\{0,\dots,n\}$, then $\pi_i\pi_{i+1}<0$ is equivalent to $\pi'(|\pi_i|)=-|\pi_{i+1}|<0$,
and it follows that $\sch(\pi)=\neg(\pi')$. This proves the first equality of the theorem.

The case of $R_n$ is slightly different. In this proof, we will denote by $\pi_i^+$ the number
$\pi_i+1$ if $\pi_i>0$, and $\pi_i-1$ if $\pi_i<0$. Let $\pi\in\mathcal{S}^{00}_{n+1}$. If $n+1$ 
is odd, hence $\pi_{n+1}>0$, we define 
$\pi'=(1,\pi_1^+,\dots,\pi_{n+1}^+)(-1,-\pi_1^+,\dots,-\pi_{n+1}^+)$. If $n+1$ is even, 
hence $\pi_{n+1}<0$, 
we define $\pi'=(1,\pi_1^+,\dots,\pi_{n+1}^+,-1,-\pi_1^+,\dots,-\pi_{n+1}^+  )$. In each case, 
it is easily checked that this defines a bijection from $\mathcal{S}^{00}_{n+1}$ to the subset 
of $\pi\in\mathcal{C}^\circ_{n+2}$ satisfying $\pi_1>0$, so that a change of sign in $\pi$ correspond 
to a $i>0$ with $\pi'(i)<0$.

The case of $Q_n$ requires more attention. Let $\pi\in\mathcal{S}^0_n$, we need to split it into blocks that will
correspond to the cycles of an element $\pi'\in\mathcal{C}_n$. This can be done by a variant of Foata's fundamental
transform \cite{lothaire}. Let $a_1<\dots<a_k$ be the left-to-right maxima of $|\pi_1|\dots|\pi_n|\in\mathfrak{S}_n$
(recall that $u$ is a left-to-right maximum of $\sigma\in\mathfrak{S}_n$ when $\sigma(i)<\sigma(u)$ for any $i<u$).
We consider the factors of $\pi$, $f_1=\pi_{a_1}\dots\pi_{a_2-1}$, $f_2=\pi_{a_2}\dots\pi_{a_3-1}$, up to
$f_k=\pi_{a_k}\dots\pi_{n}$ (we can take the convention $a_{k+1}=n+1$).
We form $\pi'$ by putting together cycles, so that if $f_i$ has odd length it gives a one-orbit cycle
$(\pi_{a_i}\dots\pi_{a_{i+1}-1},-\pi_{a_i}\dots {-\pi_{a_{i+1}-1}} )$, and if it has even length it gives a two-orbit
cycle $(\pi_{a_i}\dots\pi_{a_{i+1}-1})(-\pi_{a_i}\dots {-\pi_{a_{i+1}-1}} )$. The definition of $a_i$ shows that
$\pi_{a_i}$ is greater than $\pi_{a_i+1}\dots\pi_{a_{i+1}-1}$, and consequently these
cycles are indeed alternating. The inverse bijection is easily deduced:
we can write $\pi'\in\mathcal{C}_n$ as a product of cycles $\pi'=c_1\dots c_k$, where the cycles are sorted
so that their maximal entries are increasing. To each $c_i$ we associate a word $f_i$ so that
$c_i$ is either $(f_i,-f_i)$ or $(f_i)(-f_i)$, and such that the first letter of $f_i$ is the maximum entry of $c_i$.
Then we can form the snake $\pi=f_1,\epsilon_2f_2,\dots, \epsilon_kf_k$, where the signs 
$\epsilon_2,\dots,\epsilon_k\in\{\pm1\}$ are the unique ones such that $\pi$ is alternating.

It is in order to give examples of the previous two bijections. If we start from a snake 
$\pi=(0),4,-2,-1,-5,3,-6,(7)\in\mathcal{S}_6$, the left-to-right maxima of $|\pi|$ are $1,4,6$, and this gives 
$\pi'=(4,-2,-1,-4,2,1)(-5,3)(5,-3)(6,-6)\in\mathcal{C}_6$. As for the other direction, let 
$\pi'=(5,-2,3,-5,2,-3)(6,-1)(-6,1)(7,-4)(-7,4)\in\mathcal{C}_7$. It is a product of three cycles so that 
the snake $\pi$ is formed by putting together three words $5,-2,3$, and $6,-1$, and $7,-4$. We need
to change the signs of the second and third words to obtain a snake, which is $\pi=(0),5,-2,3,-6,1,-7,4,(-8)$.

One can check on the examples that in both cases we have $\sch(\pi)=\neg(\pi')$, but this does not
immediately follows the construction and need to be proved now. Let $a,b>0$, such that $\pi'(a)=-b$, 
we can associate a sign change in $\pi$ to each such pair $(a,b)$. The two integers are in the same cycle $c_i$.
If the word $f_i$ is of the form $\dots,a,-b,\dots$ or $\dots,-a,b,\dots$, then either the factor $a,-b$ or 
$-a,b$ appears in the word $\epsilon_if_i$, and consequently appears in $\pi$ as a sign change. The other 
possibility is that $f_i=b,\dots,-a$ which only occur when $f_i$ has even length, and $f_i=b,\dots,a$ which 
only occur when $f_i$ has odd length. If $b$ (respectively $-b$) appear in $\pi$, it is greater (respectively, 
smaller) than his neighbor entries, and after examining a few cases it comes that the same is true for $a$.
It follows that there is a sign change in $\pi$ between the last entry of $\epsilon_i f_i$ and its right neighbor.
This completes the proof.
\end{proof}

Using cycle-alternating permutations, there is a bijective proof 
of the fact that $P_{n+1}=(1+t^2)R_n$. This can be written
\begin{equation} \label{relpq2}
\sum_{\substack{\pi\in\mathcal{C}^\circ_{n+2} }} t^{\neg(\pi)}
= (1+t^2) \sum_{\substack{\sigma\in\mathcal{C}^\circ_{n+2} \\ \sigma_1>0}} 
t^{\neg(\sigma)}.
\end{equation}
Now, consider the conjugation by the transposition $(1,-1)$, {\it i.e.} the map 
that sends $\pi$ to $(1,-1)\pi (1,-1)$. This is a fixed-point-free involution on
 $\mathcal{C}^\circ_{n+2}$
which proves combinatorially \eqref{relpq2}, as is easily seen on the arch diagrams. Indeed, 
in this representation the involution exchanges the dot labelled $1$ with the dot labelled $-1$, 
see Figure~\ref{conjug}. We have $\pi(1)>0$ if and only if its image $\pi'=(1,-1)\pi (1,-1)$
does not satisfy $\pi'(1)>0$, and if $\pi(1)>0$ there holds $\neg(\pi')=\neg(\pi)+2$.
This proves \eqref{relpq2}. Note that it also is possible to prove $P_{n+1}=(1+t^2)R_n$ on
the snakes, using Proposition~\ref{propval}, but this is more tedious.

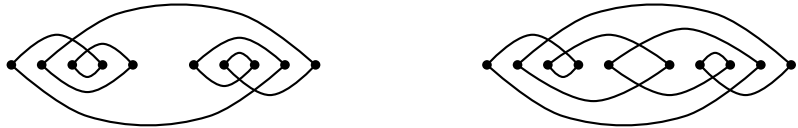
\begin{figure}[h!tp] \psset{unit=4mm}
\begin{pspicture}(-5,-2.5)(5,2.5)
\psdots(-1,0)(-2,0)(-3,0)(-4,0)(-5,0)
\psdots(1,0)(2,0)(3,0)(4,0)(5,0)
\pscurve(-5,0)(-3.5,1)(-2,0)  \pscurve(5,0)(3.5,-1)(2,0)
\pscurve(-3,0)(-2,0.7)(-1,0)    \pscurve(3,0)(2,-0.7)(1,0)
\pscurve(-4,0)(-1.5,1.7)(2.5,1.7)(5,0) \pscurve(4,0)(1.5,-1.7)(-2.5,-1.7)(-5,0)
\pscurve(-4,0)(-2.5,-0.9)(-1,0)   \pscurve(4,0)(2.5,0.9)(1,0)
\pscurve(-3,0)(-2.5,-0.4)(-2,0)   \pscurve(3,0)(2.5,0.4)(2,0)
\end{pspicture}
\hspace{2cm}
\begin{pspicture}(-5,-2.5)(5,2.5)
\psdots(-1,0)(-2,0)(-3,0)(-4,0)(-5,0)
\psdots(1,0)(2,0)(3,0)(4,0)(5,0)
\pscurve(-5,0)(-3.5,1)(-2,0)  \pscurve(5,0)(3.5,-1)(2,0)
\pscurve(-3,0)(-1,1)(1,0)    \pscurve(3,0)(1,-1)(-1,0)
\pscurve(-4,0)(-1.5,1.7)(2.5,1.7)(5,0) \pscurve(4,0)(1.5,-1.7)(-2.5,-1.7)(-5,0)
\pscurve(-4,0)(-1.5,-1.2)(1,0)   \pscurve(4,0)(1.5,1.2)(-1,0)
\pscurve(-3,0)(-2.5,-0.4)(-2,0)   \pscurve(3,0)(2.5,0.4)(2,0)
\end{pspicture}
\caption{\label{conjug}
The arch diagram of the cycle-alternating permutation $\pi=(1$, $4$, $-5$, $-2$, $-3$, $-1$, $-4$, $5$, $2$, $3)
\in\mathcal{C}^\circ_{5}$, and of its conjugate $(1,-1)\pi (1,-1)$.}
\end{figure}

Using the cycle-alternating permutations, we can also obtain the exponential generating
function of $\{P_n\}_{n\geq0}$, knowing the one of $\{Q_n\}_{n\geq0}$. Indeed, a cycle-alternating
permutation is an assembly of cycle-alternating cycles, in the sense of combinatorial
species, and this gives immediately:
\begin{equation}
  \sum_{n=0}^\infty P_n(t) \frac{z^{n+1}}{(n+1)!} 
  = \log\left( \sum_{n=0}^\infty Q_n(t) \frac{z^n}{n!} \right)
  = \log\left(\frac{1}{\cos z - t\sin z} \right),
\end{equation}
and after differentiation:
\begin{equation}
  \sum_{n=0}^\infty P_n(t) \frac{z^n}{n!} = \frac{\sin z + t\cos z}{\cos z - t\sin z}.
\end{equation}

\subsection{The $q$-analogs} It is possible to give combinatorial interpretations of $Q_n(t,q)$ and 
$R_n(t,q)$ using cycle-alternating permutations, and the notion of {\it crossing} for signed permutations
which was defined in \cite{CJW}. But in the case of cycle-alternating permutations, there is a more simple
equivalent definition which is the following.

\begin{defn}
A {\it crossing} of $\pi\in\mathcal{C}_n$ is a pair $(i,j)\in\llbracket n \rrbracket^2$ such that $i<j<\pi(i)<\pi(j)$.
Equivalently, it is the number of intersections of two arches above the horizontal axis in the arch diagram of
$\pi$. We denote by $\cro(\pi)$ the numbers of crossings in $\pi$.
\end{defn}

Because of the symmetry in the arch diagram, $\cro(\pi)$ is also half the total number of intersections
between two arches.
For example, the two permutations in Figure~\ref{conjug} have respectively 2 and 3 crossings.
This kind of statistic can be followed through bijections between permutations and paths \cite{Corteel},
but let us first show that $Q_n(t,q)$ and $R_n(t,q)$ are related with some weighted Dyck prefixes 
(a Dyck prefix being similar to a Dyck path except that the final height can be non-zero).

\begin{defn}
Let $\mathcal{P}_n$ be the  set of weighted Dyck prefixes of length $n$, where
\begin{itemize}
 \item each step $\nearrow$ between heights $h$ and $h+1$ has a weight $q^i$ for some $i\in\{0,\dots h+1\}$,
 \item each step $\searrow$ between heights $h$ and $h+1$ has a weight $q^i$ for some $i\in\{0,\dots h\}$.
\end{itemize}
Let $\mathcal{P}'_n\subset\mathcal{P}_n$ be the subset of paths $p$ such that there is no step $\nearrow$
starting at height $h$ with the maximal weight $q^{h+1}$. For any $p\in\mathcal{P}_n$, let $\fh(p)$ be its
final height, and let $\w(p)$ be its total weight, {\it i.e.} the product of the weights of each step.
\end{defn}

\begin{prop} \label{dyckpre}
We have:
\begin{equation}
Q_n(t,q)= \sum_{p \in \mathcal{P}'_n } t^{\fh(p)} q^{\w(p)}, \qquad 
R_n(t,q)= \sum_{p \in \mathcal{P}_n } t^{\fh(p)} q^{\w(p)}.
\end{equation}
\end{prop}

\begin{proof}
In the proof of Proposition~\ref{frac_pq}, we have used the matrices
$D_2+U_2D_2U_2$ and $D_2+D_2U_2U_2$, thought of as transfer-matrices for walks in 
the non-negative integers. We can do something similar with the matrices 
\begin{equation}
  D_1+U_1D_1U_1=\left(\begin{mmatrix}
    0    & [1]_q  &   &  (0) \\
   [1]_q & 0 &  [2]_q     \\
         & [2]_q  & 0   &  \bla \\ 
   (0)   &        & \bla  &  \bla \\
  \end{mmatrix}\right), \;
  D_1+D_1U_1U_1=\left(\begin{mmatrix}
    0    & [1]_q  &   &  (0) \\
   [2]_q & 0 &  [2]_q     \\
         & [3]_q  & 0   &  \bla \\ 
   (0)   &        & \bla  &  \bla \\
  \end{mmatrix}\right).
\end{equation}
Recall that we have \eqref{rel1}, where $W_1=(1,t,t^2,\dots)$. In terms of paths, the coefficients in the matrices and
vectors have the following meaning: from $V_1=(1,0,\dots)^*$, we only consider paths starting at height 0, and
from $W_1=(1,t,t^2,\dots)$ there is the weight $t^{\fh(p)}$ for each path $p$. The coefficients $(h+1,h)$ in the matrices
give the possible weights on steps $\nearrow$ from height $h$ to $h+1$, the coefficients $(h,h+1)$ in the matrices
give the possible weights on steps $\searrow$ from height $h+1$ to $h$. The result follows.
\end{proof}

\begin{prop}
There is a bijection $\Psi:\mathcal{C}_n\to\mathcal{P}'_n$ such that $\neg(\pi)=\fh(p)$ and $\cro(\pi)=\w(p)$ if the image
of $\pi$ is $p$.
\end{prop}

\begin{proof}
We can use the the bijection $\Psi_{FZ}$ from \cite{Corteel} between permutations and weighted Motzkin paths.
To do this, we identify $\mathfrak{S}^\pm_n$ to a subset of $\mathfrak{S}_{2n}$ via the order-preserving bijection
$\llbracket n\rrbracket \to [2n]$. This subset is characterized by the fact that the arrow diagrams are centrally
symmetric, and via $\Psi_{FZ}$ from \cite{Corteel}, it is in bijection with some weighted Motzkin paths that are
vertically symmetric. Cycle-alternating permutations correpond to the case where there is no horizontal step,
{\it i.e.} they are in bijection with some weighted Dyck paths that are vertically symmetric. Keeping the first half
of these Dyck paths gives the desired bijection.
\end{proof}

For convenience, we rephrase here explicitly the bijection of the previous proposition. Let $\pi\in\mathcal{C}_n$,
then we define a $\Psi(\pi)\in\mathcal{P}'_n$ such that, for any $j\in\{{-n}\dots{-1}\}$:
\begin{itemize}
 \item the $(n+1+j)$th step is $\nearrow$ if $\pi^{-1}(j)>j<\pi(j)$ and $\searrow$ if $\pi^{-1}(j)<j>\pi(j)$,
 \item if the $(n+1+j)$th step is $\nearrow$, it has a weight $q^k$ where $k$ is the number of $i$ such that $(i,j)$ is a
         crossing, {\it i.e.} $i<j<\pi_i<\pi_j$,
 \item if the $(n+1+j)$th step is $\searrow$, it has a weight $q^k$ where $k$ is the number of $i$ such that $(-i,-j)$ is a
         crossing, {\it i.e.} $-i<-j<-\pi_i<-\pi_j$.
\end{itemize}
See Figure~\ref{exbijcro} for an example.

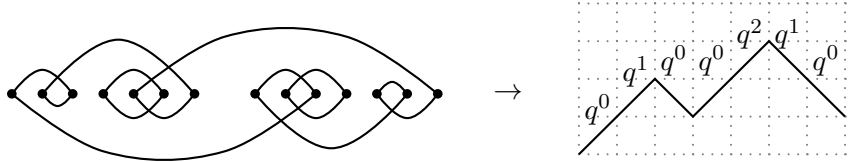
\begin{figure}[h!tp] \psset{unit=4mm}
\begin{pspicture}(-7,-2)(7,2)
\psdots(-7,0)(-6,0)(-5,0)(-4,0)(-3,0)(-2,0)(-1,0)\psdots(7,0)(6,0)(5,0)(4,0)(3,0)(2,0)(1,0)
\pscurve(-7,0)(-6,0.8)(-5,0)\pscurve(-6,0)(-3.5,1.8)(-1,0)\pscurve(-4,0)(-3,0.8)(-2,0)\pscurve(-3,0)(0,1.9)(4,1.9)(7,0)
\pscurve(7,0)(6,-0.8)(5,0)\pscurve(6,0)(3.5,-1.8)(1,0)\pscurve(4,0)(3,-0.8)(2,0)\pscurve(3,0)(0,-1.9)(-4,-1.9)(-7,0)
\pscurve(1,0)(2,0.8)(3,0)\pscurve(2,0)(3,0.8)(4,0)\pscurve(5,0)(5.5,0.4)(6,0)
\pscurve(-1,0)(-2,-0.8)(-3,0)\pscurve(-2,0)(-3,-0.8)(-4,0)\pscurve(-5,0)(-5.5,-0.4)(-6,0)
\end{pspicture}
\begin{pspicture}(-2,-2)(2,2)
\rput(0,0){$\to$}
\end{pspicture}
\psset{unit=5mm}
\begin{pspicture}(0,0)(7,4)
\psgrid[gridcolor=gray,griddots=4,subgriddiv=0,gridlabels=0](0,0)(7,4)
\psline(0,0)(2,2)(3,1)(5,3)(7,1)
\rput(0.5,1.2){$q^0$}\rput(1.5,2.2){$q^1$}\rput(2.5,2.5){$q^0$}
\rput(3.5,2.5){$q^0$}\rput(4.5,3.2){$q^2$}\rput(5.5,3.2){$q^1$}\rput(6.5,2.5){$q^0$}
\end{pspicture}
\caption{ The bijection $\Psi$ from $\mathcal{C}_n$ to $\mathcal{P}'_n$ in the case of
  $\pi=3,4,-7,2,6,1,5$.  \label{exbijcro}}
\end{figure}

Through this bijection we obtain combinatorial models of $Q_n(t,q)$ and $R_n(t,q)$ using the notion of crossing.

\begin{prop} We have
\begin{equation}
Q_n(t,q) = \sum_{\pi\in\mathcal{C}_n} t^{\neg(\pi)}q^{\cro(\pi)}  , \qquad
R_n(t,q) = \sum_{\substack{ \pi\in\mathcal{C}_{n+1} \\ \pi_{n+1}<0 }} 
   t^{\neg(\pi)-1}q^{\cro(\pi)}.
\end{equation}
Alternatively, the condition $\pi_{n+1}<0$ in the second identity can be replaced with $\pi^{-1}(n+1)<0$.
\end{prop}

\begin{proof}
The first identity follows from Proposition~\ref{dyckpre} and the bijection $\Psi$ between $\mathcal{C}_n$
and $\mathcal{P}'_n$, but the second one is not as immediate.

First, note that the bijection $\pi\mapsto\pi^{-1}$ stabilizes the set $\mathcal{C}_n$ and preserves the number
of crossings, since it is just an horizontal symmetry on the arch diagrams. This proves the fact that we can
replace $\pi_{n+1}<0$ with $\pi^{-1}(n+1)<0$.

Then, let us consider the subset $\mathcal{P}''_{n+1}\subset\mathcal{P}'_{n+1}$ of elements $p$ such that
there is no step $\searrow$ from height $h+1$ to $h$ with weight $q^h$ (note that this condition implies there
is no return to height $0$). There is an obvious bijection from  $\mathcal{P}''_{n+1}$ to $\mathcal{P}_{n}$,
because removing the first step in $p\in\mathcal{P}''_{n+1}$ gives a path which is in $\mathcal{P}_{n}$ with
respect to the shifted origin $(1,1)$. This proves:
\begin{equation}
  R_n(t,q) = \sum_{p\in\mathcal{P}''_{n+1}} t^{\fh(p)-1} q^{\w(p)}.
\end{equation}
It remains to show that the image of $\mathcal{P}''_{n+1}$ via $\Psi^{-1}$ is precisely the set of $\pi\in\mathcal{C}_{n+1}$
such that $\pi^{-1}(n+1)<0$. Essentially, this follows from the fact that there are some steps in the path which characterize
the right-to-left minima in $\pi$, and we can use Lemma~3.2.2 from \cite{Josuat2}. In our case, the subset
$\mathcal{P}''_{n+1}\subset\mathcal{P}'_{n+1}$ corresponds to a subset of $\mathfrak{S}_{2n+2}$ (via the identification
$\mathfrak{S}^\pm_{n+1}\to\mathfrak{S}_{2n+2}$), more precisely to the subset of $\sigma$ having no right-to-left minima
among $1\dots n+1$, {\it i.e.} $\sigma^{-1}(1)>n+1$. In terms of $\pi\in\mathfrak{S}^\pm_{n+1}$, this precisely
means that $\pi^{-1}(-n-1)>0$ and completes the proof.
\end{proof}

\begin{rem}
Independently from our work, Chen, Fan and Jia \cite{CFJ} gave a bijection between snakes of type $B$ and Dyck
prefixes such that there are $h+1$ possible choices on each step between height $h$ and $h+1$, and their statistic
$\alpha(\pi)$ (see Theorem~4.5) is the same as our statistic $\sch(\pi)$, though defined differently. We can have such
a bijection by composing two we have presented, but their bijection is not the same as ours and does not seem to be
directly related. To advocate our bijection $\mathcal{S}_n\to\mathcal{P}'_n$, note that going through the set $\mathcal{C}_n$
we just had to adapt two known bijections: Foata's fundamental transform \cite{lothaire} (in the direction
$\mathcal{S}_n\to\mathcal{C}_n$), then Corteel's bijection $\Psi_{FZ}$ \cite{Corteel} originally given by Foata and Zeilberger
(in the direction $\mathcal{C}_n\to\mathcal{P}'_n$).
\end{rem}

\subsection{Another variant of $\Psi_{FZ}$}
With another adaptation of this bijection, we can link $\mathcal{C}_n$ and some weighted Motzkin path, so that we
obtain a combinatorial proof of a $J$-fraction generalizing \eqref{fvfrac1} and \eqref{fvfrac2} in Theorem~\ref{jfracqnab}
below. As a consequence of Theorem~\ref{comb_cycl}, we have:
\begin{equation} \label{qnatcomb}
  Q_n^{(a)}(t)=\sum_{\pi\in\mathcal{C}_n} a^{{\rm cyc}(\pi)} t^{\neg(\pi)}
\end{equation}
where ${\rm cyc}(\pi)$ is the number of cycles, indeed since there is a ``cycle structure'' in $\mathcal{C}_n$ it suffices to take
the $a$th power of the exponential generating function to have a parameter $a$ counting ${\rm cyc}(\pi)$.
With the recurrence for $Q_n^{(a)}(t)$ in \eqref{rec_pqa}
and with the same method as in the case of $Q_n$ and $R_n$, we can prove a continued fraction for
the ordinary generating function of $\{Q_n^{(a)}(t)\}_{n\geq0}$.
We prove here bijectively this continued fraction, by going from $\mathcal{C}_n$ to weighted Motzkin paths.
To this end, we use another representation of a signed permutation, the {\it signed arch diagram}. Let
$\pi\in\mathfrak{S}^{\pm}_n$, this diagram is obtained the following way: draw $n$ dots labelled by $1\dots n$
from left to right on the horizontal axis, then for each $i\in[n]$, draw an arch from $i$ to $|\pi_i|$, and label
this arrow with a $+$ if $\pi_i>0$ and with a $-$ if $\pi_i<0$. We understand that as in the previous arch
notation, the arch is above the axis if $i\leq|\pi_i|$ and below otherwise. See the left part of Figure~\ref{signarrow}
for an example. (In the case $|\pi_i|=i$, the arch is a loop attached to the dot $i$.) A case-by-case study shows that
$\pi$ is cycle-alternating if and only if it avoids certain configurations in the signed arch diagram, which are listed
in the right part of Figure~\ref{signarrow} ($\pm$ being either $+$ or $-$). These forbidden configurations are of
course reminiscent of \eqref{signs}.

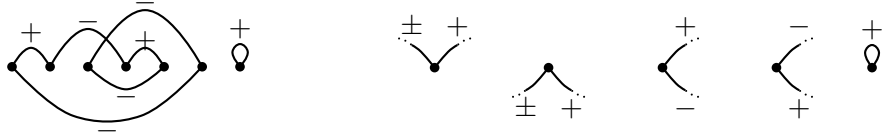
\begin{figure}[h!tp] \psset{unit=5mm}
\begin{pspicture}(1,-1.7)(6,1.7)
\psdots(1,0)(2,0)(3,0)(4,0)(5,0)(6,0)(7,0)
\pscurve(1,0)(1.5,0.5)(2,0) \rput(1.5,0.9){$+$}
\pscurve(2,0)(3,1)(4,0)      \rput(3,1.2){$-$}
\pscurve(3,0)(4.5,1.5)(6,0)  \rput(4.5,1.7){$-$}
\pscurve(4,0)(4.5,0.5)(5,0)  \rput(4.5,0.7){$+$}
\pscurve(1,0)(2.6,-1.3)(4.4,-1.3)(6,0) \rput(3.5,-1.7){$-$}
\pscurve(3,0)(4,-0.6)(5,0)  \rput(4,-0.8){$-$}
\pscurve(7,0)(6.8,0.4)(7,0.6)(7.2,0.4)(7,0) \rput(7,1){$+$}
\end{pspicture}
\hspace{2.5cm} \psset{unit=3mm}
\begin{pspicture}(-1,-2.8)(1,2.8)
 \psdot(0,0)
\pscurve(0,0)(0.7,0.8)(1,1)\pscurve[linestyle=dotted,dotsep=0.6mm](1,1)(1.3,1.2)(1.6,1.3) \rput(1,1.8){$+$}
\pscurve(0,0)(-0.7,0.8)(-1,1)\pscurve[linestyle=dotted,dotsep=0.6mm](-1,1)(-1.3,1.2)(-1.6,1.3) \rput(-1,1.8){$\pm$}
\end{pspicture}
\qquad
\begin{pspicture}(-1,-2.8)(1,2.8)
 \psdot(0,0)
\pscurve(0,0)(0.7,-0.8)(1,-1)\pscurve[linestyle=dotted,dotsep=0.6mm](1,-1)(1.3,-1.2)(1.6,-1.3) \rput(1,-1.8){$+$}
\pscurve(0,0)(-0.7,-0.8)(-1,-1)\pscurve[linestyle=dotted,dotsep=0.6mm](-1,-1)(-1.3,-1.2)(-1.6,-1.3) \rput(-1,-1.8){$\pm$}
\end{pspicture}
\qquad
\begin{pspicture}(-1,-2.8)(1,2.8)
\psdot(0,0)
\pscurve(0,0)(0.7,0.8)(1,1)\pscurve[linestyle=dotted,dotsep=0.6mm](1,1)(1.3,1.2)(1.6,1.3) \rput(1,1.8){$+$}
\pscurve(0,0)(0.7,-0.8)(1,-1)\pscurve[linestyle=dotted,dotsep=0.6mm](1,-1)(1.3,-1.2)(1.6,-1.3) \rput(1,-1.8){$-$}
\end{pspicture}
\qquad
\begin{pspicture}(-1,-2.8)(1,2.8)
\psdot(0,0)
\pscurve(0,0)(0.7,0.8)(1,1)\pscurve[linestyle=dotted,dotsep=0.6mm](1,1)(1.3,1.2)(1.6,1.3) \rput(1,1.8){$-$}
\pscurve(0,0)(0.7,-0.8)(1,-1)\pscurve[linestyle=dotted,dotsep=0.6mm](1,-1)(1.3,-1.2)(1.6,-1.3) \rput(1,-1.8){$+$}
\end{pspicture}
\qquad
\begin{pspicture}(-0.2,-2.8)(0.2,2.8)
 \psdot(0,0)
\pscurve(0,0)(-0.3,0.7)(0,1)(0.3,0.7)(0,0) \rput(0,1.7){$+$}
\end{pspicture}
\caption{The signed arch diagram of $\pi=2,-4,-6,5,-3,-1,7$, and the forbidden configurations in the signed
arch diagram of a cycle-alternating permutation. \label{signarrow} }
\end{figure}

Once we know the forbidden configurations, it is possible to encode the signed arch diagram of
$\pi\in\mathcal{C}_n$ by a weighted Motzkin path of $n$ steps, in a way similar to $\Psi_{FZ}$
from \cite{Corteel}. The diagram of $\pi$ is scanned from
left to right, and the path is built such that:
\begin{itemize}
 \item If the $i$th node in $\pi$ is \nodr then the $i$th step of $p$ if $\nearrow$, moreover this step has label
$+$ (resp. $-$) if the two arches starting from the $i$th node have label $+$ (resp. $-$).
 \item If the $i$th node in $\pi$ is \nodh then the $i$th step of $p$ if $\rightarrow$. Moreover this step has label
$(j)$ if the left strand in the $i$th node is connected to the $j$th strand of the ``partial'' 
signed arch diagram (say, from bottom to top). See for example the left part of Figure~\ref{bijqa} where there are
three possible choices to connect this kind of node. We can see that the possible labels are $(1),\dots,(h)$
where $h$ is the starting height of the $i$th step in the path.
 \item Similarly if the $i$th node in $\pi$ is \nodb then the $i$th step of $p$ if $\rightarrow$, and this step has label
$(j')$ if the $i$th node is connected on the left to the $j$th strand of the ``partial'' signed arch diagram (say, from top
to bottom).
\item If the $i$th node in $\pi$ is \nodf then the $i$th step is $\rightarrow$ with label $(0)$.
\item If the $i$th node in $\pi$ is \nodg then the $i$th step is $\searrow$ with a label $(j,k)$ where $j$ and $k$
        respectively encode where the upper and lower strands are connected, as in the case of \nodh and \nodb
        above.
\end{itemize}

\begin{figure}[h!tp]\psset{unit=4mm}
\begin{pspicture}(1,-3.5)(10,3.5)
\psdots(1,0)(2,0)(3,0)(4,0)(5,0)(6,0)(7,0)(8,0)
\pscurve(1,0)(1.5,0.5)(2,0) \rput(1.5,1){$+$}
\pscurve(2,0)(3,1)(4,0)      \rput(3,1.2){$-$}
\pscurve(3,0)(4.5,2.1)(7,3.3)(9,3.5) \rput(4.5,2.9){$+$}
\pscurve(4,0)(5.5,1.4)(7,0) \rput(5.2,1.6){$-$}
\pscurve(5,0)(7,1.9)(9,2.3) \rput(7.5,2.4){$-$}
\pscurve(5,0)(7,-1.9)(9,-2.3) \rput(7.5,-2.4){$-$}
\pscurve(6,0)(6.5,-0.6)(7,0) \rput(6.5,-0.8){$-$}
\pscurve(8,0)(8.5,0.7)(9,0.8) \rput(8.7,1.2){$+$}
\pscurve(8,0)(8.5,-0.7)(9,-0.8) \rput(8.7,-1.2){$+$}
\pscurve(3,0)(4.5,-1.4)(6,0) \rput(4.5,-1.9){$+$}
\pscurve(1,0)(2.5,-2)(4.7,-3)(9,-3.5) \rput(4,-3.4){$+$}
\rput(10,0){$\dots$}
\end{pspicture} \psset{unit=4mm}
\begin{pspicture}(-3,-3.5)(3,3.5)  \rput(0,0){$\to$}  \end{pspicture}\psset{unit=6mm}
\begin{pspicture}(0,-0.5)(9,4.5)
\psgrid[gridcolor=gray,griddots=4,subgriddiv=0,gridlabels=0](0,0)(9,4)
\psline(0,0)(1,1)(2,1)(3,2)(4,2)(5,3)(6,3)(7,2)(8,3)
\rput(0.5,1.3){\footnotesize $+$}
\rput(1.5,1.6){\footnotesize $(1)$}
\rput(2.5,2.3){\footnotesize $+$}
\rput(3.5,2.6){\footnotesize $(2)$}
\rput(4.5,3.3){\footnotesize $-$}
\rput(5.5,3.8){\footnotesize $(2')$}
\rput(6.5,3.3){\footnotesize (2,1)}
\rput(7.5,3.3){\footnotesize $+$}
\rput(9.2,2.3){$\dots$}
\end{pspicture}
\caption{ Construction of the bijection from the signed arch diagram of $\pi\in\mathcal{C}_n$ to
a labelled Motzkin path.  \label{bijqa}}
\end{figure}
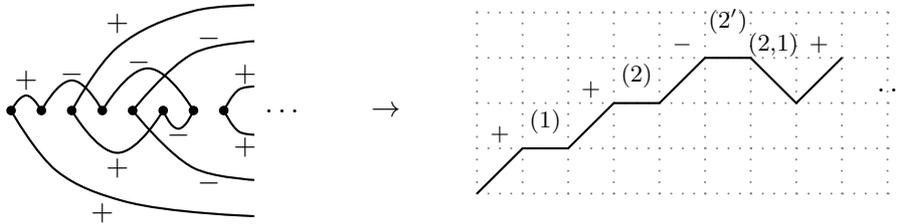

Quite a few bijections of this kind are known so we will not give further details. As for the parameters
$a$ and $t$, they are taken into account the following way:
\begin{itemize}
 \item For each arch with label $-$ linking $i$ and $j$ where $i<j$, we give a weight $t$ to the $i$th step
         in the path. Thus there is a weight $t^2$ on the step $\nearrow$ with label $-$, and a weight $t$
         on each step $\rightarrow$.
\item For each $i$ which is the maximal element of a cycle, we give a weight $a$ to the $i$th step.
        Thus each step $\rightarrow$ with weight $(0)$ has a weight $a$, since these correspond to
         nodes \nodf or cycles of the form $(i,-i)$. Besides, there is a weight $a$ on some of the step
         $\searrow$. More precisely, there are $h^2$ possible labels on a step $\searrow$ starting at 
         height $h$, and $h$ of them have this weight $a$. Indeed, for each choice of where to connect
         the upper strand of the node \nodg there is a unique choice of where to connect the lower strand
         so that a new cycle is ``closed''.
\end{itemize}
All in all, using the classical correspondence between weighted Motzkin paths and $J$-fractions,
this bijection proves the following theorem.

\begin{thm} \label{jfracqnab}
Let $Q_n^{(a)}(t)$ be the polynomials in $a$ and $t$ with exponential
generating function $(\cos z - t\sin z)^{-a}$, or defined combinatorially by \eqref{qnatcomb}. Then,
using the notation in Definition~\ref{def_jfra}, we have:
\begin{equation}
  \sum_{n=0}^\infty Q_{n}^{(a)}(t) z^n = \mathfrak{J}\big(  \;  (2h+a)t , \; h(h-1+a)(1+t^2)  \;  \big).
\end{equation}
\end{thm}

Apart the combinatorial side, this result was already known by Stieltjes~\cite{stieltjes}.
Besides his analytical proof, it is also possible to obtain the continued fraction by showing that
the exponential generating function satisfies some particular kind of addition formula which
permits to make use of a theorem of Stieltjes and Rogers, see for example \cite[Theorem 5.2.10]{GJ}.
We thank Philippe Flajolet for communicating these facts and references.

\section{Increasing trees and forests} 
\label{sectree}

So far, we have only used recurrence relations and bijections to derive our new models of the derivative polynomials.
Some more elaborate methods apply pretty well in our case to give combinatorial models in terms of increasing trees
and forests. We first present how to use the combinatorial theory of differential equations, as exposed by Leroux and
Viennot in \cite{LV}, and secondly how to use recent results of Błasiak and Flajolet related with operators and normal
ordering \cite{Fla}.

\subsection{Increasing trees via differential equations}
An archetype example in the combinatorial theory of differential equation is the one of the tangent and secant functions
(see \cite{LV}), and it has become a classical method to show that Euler numbers $E_n$ count some increasing trees,
{\it i.e.} labelled trees where labels are increasing from the root to the leaves. Here we have a system of two differential
equations similar to the one of tangent and secant, except that an initial condition is given by the parameter $t$ instead
of $0$. This will give rise to more general trees where we allow some leaves with no label having a weight $t$.

\begin{lem} Let $f=\sum P_n z^n /n!$ and $g=\sum Q_n z^n /n!$, then we have:
\begin{equation}
  \begin{cases} f' =  1 + f^2 & f(0) = t, \\
    g' = fg          & g(0) = 1.
  \end{cases}
\end{equation}
\end{lem}
\begin{proof}
Of course, this can be checked on the closed form given in \eqref{genpq}. Also, $f'=1+f^2$ can be directly checked
on the snakes as in the case of alternating permutations \cite{Sta}, and $g'=fg$ follows from $f=(\log g)'$ previously seen.
\end{proof}

It is adequate to rewrite the equations in the following way:
\begin{align}
  \begin{cases} f = t+z+ \int f^2,  \\
    g = 1 + \int fg.
  \end{cases}
\end{align}
Let us begin with the case of $f(z)$. From $f = t+z+ \int f^2$, and proceeding as in \cite{LV},
$f(z)$ counts increasing trees that are recursively produced by the following rules, starting from an 
isolated node marked by $f$:
\begin{itemize}
\item a node marked by $f$ can become a leaf with no label (this corresponds to the term $t$ in 
      $f = t+z+ \int f^2$, $f(z)$, so these leaves will have a weight $t$),
\item a node marked by $f$ can become a leaf with an integer label, this label being the smallest integer 
      that does not already appear in the tree (this corresponds to the term $z$ in $f = t+z+ \int f^2$),
\item a node marked by $f$ can become an internal node having an integer label and two (ordered) sons marked by $f$ 
      (this corresponds to the last term $\int f^2$ in $f = t+z+ \int f^2$).
      As before, the integer label is the smallest integer that does not already appear in the tree.
\end{itemize}
More precisely, the coefficient of $z^n$ in $f(z)$ counts these trees having $n$ integer labels. See 
Figure~\ref{rectrees} for an example of tree produced by these rules.

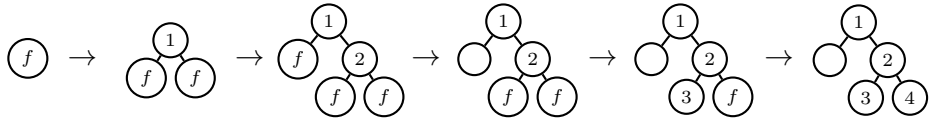
\begin{figure}[h!tp] \psset{unit=4mm}
\begin{pspicture}(30.5,5) 
\rput(0.7,0.4){
 \pstree{\Tcircle{\hbox{\tiny$f$}}} {}
}
\rput(2.5,2.5){$\rightarrow$}
\rput(5.4,2.5){
\pstree[treesep=1mm,levelsep=5mm]{ \Tcircle{\hbox{\tiny 1}} }{ \Tcircle{\hbox{\tiny $f$}} \Tcircle{\hbox{\tiny $f$}}   } 
}
\rput(8,2.5){$\rightarrow$}
\rput(11,2.5){
\pstree[treesep=3mm,levelsep=5mm]{ \Tcircle{\hbox{\tiny 1}} }{ \Tcircle{\hbox{\tiny $f$}}  {}
\pstree[treesep=1mm,levelsep=5mm]{ \Tcircle{\hbox{\tiny 2}} }{ \Tcircle{\hbox{\tiny $f$}} \Tcircle{\hbox{\tiny $f$}} } } 
}
\rput(13.8,2.5){$\rightarrow$}
\rput(16.8,2.5){
\pstree[treesep=3mm,levelsep=5mm]{ \Tcircle{\hbox{\tiny 1}} }{ \Tcircle{\hbox{\phantom{\tiny i}}}  {}
\pstree[treesep=1mm,levelsep=5mm]{ \Tcircle{\hbox{\tiny 2}} }{ \Tcircle{\hbox{\tiny $f$}} \Tcircle{\hbox{\tiny $f$}}  }  } 
}
\rput(19.6,2.5){$\rightarrow$}
\rput(22.6,2.5){
\pstree[treesep=3mm,levelsep=5mm]{ \Tcircle{\hbox{\tiny 1}} }{ \Tcircle{\hbox{\phantom{\tiny i}}}  {}
\pstree[treesep=1mm,levelsep=5mm]{ \Tcircle{\hbox{\tiny 2}} }{ \Tcircle{\hbox{\tiny 3}} \Tcircle{\hbox{\tiny $f$}} }  } 
}
\rput(25.4,2.5){$\rightarrow$}
\rput(28.4,2.5){
\pstree[treesep=3mm,levelsep=5mm]{ \Tcircle{\hbox{\tiny 1}} }{ \Tcircle{\hbox{\phantom{\tiny i}}}  {}
\pstree[treesep=1mm,levelsep=5mm]{ \Tcircle{\hbox{\tiny 2}} }{ \Tcircle{\hbox{\tiny 3}} \Tcircle{\hbox{\tiny 4}}   }  }
}
\end{pspicture}
\caption{Tree produced via the equation $f=t+z+\int f^2$. \label{rectrees} }
\end{figure}

The case of $g(z)$ is quite similar. Starting from an isolated node marked by $g$, each node marked by $g$
can become either a leaf with no label, or an internal node having two sons respectively marked by $f$
and $g$. Note that we need the production rules for $f$ to build a tree counted by $g$, and note also
that $g$ can produce an empty leaf which has no weight $t$, contrary to the empty leaves produced by $f$.

The two kinds of tree can also be given a non-recursive definition.

\begin{defn}
Let $\mathcal{T}_n$ be the set of complete binary trees, such that:
\begin{itemize}
\item except some leaves that are empty, the nodes are labelled by integers so that each $i\in[n]$ appears exactly once,
\item labels are increasing from the root to the leaves.
\end{itemize}
Let $\emp(T)$ be the number of empty leaves of $T\in\mathcal{T}_n$, and let $\mathcal{T}^*_n\subset\mathcal{T}_n$
be the subset of trees such that the rightmost leaf is empty.
\end{defn}
The production rules for $f$ (respectively, $g$) can be checked on the set $\mathcal{T}_n$ (respectively $\mathcal{T}^*_n$),
so that the result of the above discussion is the following.
\begin{thm} \label{thtree1}
We have:
\begin{equation}
P_n(t) = \sum_{T\in\mathcal{T}_{n}} t^{\emp(T)}, \qquad
Q_n(t) = \sum_{T\in\mathcal{T}^*_n} t^{\emp(T)-1}.
\end{equation}
\end{thm}

\subsection{Increasing trees via normal ordering}
Some recent results of Błasiak and Flajolet~\cite{Bla} directly apply to this context and also give models of
$P_n(t)$ and $Q^{(a)}_n(t)$ in terms of increasing trees and forests. Let $D$ be the derivation with respect to $t$ and $U$
the multiplication by $t$ ({\it i.e.} we are in the particular $q=1$ of operators defined in \eqref{def_du}). There holds
$DU-UD=I$. A general idea in \cite{Bla} is that the coefficients $c_{i,j}$ in the normal form of $f(D,U)$ as in
\eqref{normalorder}, at least for certain particular forms of $f(D,U)$, naturally counts some labelled directed graphs 
which are produced by connecting some ``gates''. In the present case, when $f(D,U)$ is $(D+UDU)^n$ for $P_n$
and $(D+aU+UUD)^n$ for $Q^{(a)}_n$, and we obtain some increasing trees and forests. The theorem of this
subsection is a direct application of a main result of Flajolet and Błasiak~\cite[Theorem 1]{Bla}, so as a proof we will
roughly explain some ideas leading to the definition below and refer to \cite{Bla} for more details.

\begin{defn}
Let $\mathcal{F}_n$ be the set of plane rooted forests satisfying the following conditions:
\begin{itemize}
\item each root has exactly one child, and each of the other internal nodes has exactly two (ordered) children,
\item there are $n$ nodes labelled by integers from $1$ to $n$, but some leaves can
      be non-labelled (these are called {\it empty} leaves), and labels are increasing 
      from each root down to the leaves.
\end{itemize}
Note that the trees forming a forest are unordered. Let $\mathcal{U}_n\subset\mathcal{F}_n$ be the subset of trees, {\it i.e.}
forests with one connected component. For any tree or forest $T$, let $\emp(T)$ be the number of empty leaves, and let 
$\cc(T)$ be its number of connected components.
\end{defn}

For example, there are 11 elements in $\mathcal{F}_3$ and they are:

\medskip

\noindent
\pstree[levelsep=6mm]
       { \Tcircle{\tiny 1} }
       { \Tcircle{\tiny\phantom 1} }
\pstree[levelsep=6mm]
       { \Tcircle{\tiny 2} }
       { \Tcircle{\tiny\phantom 1} }
\pstree[levelsep=6mm]
       { \Tcircle{\tiny 3} }
       { \Tcircle{\tiny\phantom 1} }\;,
\qquad
\pstree[levelsep=6mm]
       { \Tcircle{\tiny 1} }
       { \Tcircle{\tiny 2} }
\pstree[levelsep=6mm]
       { \Tcircle{\tiny 3} }
       { \Tcircle{\tiny \phantom 1} }\;,
\qquad
\pstree[levelsep=6mm]
       { \Tcircle{\tiny 1} }
       { \Tcircle{\tiny 3} }
\pstree[levelsep=6mm]
       { \Tcircle{\tiny 2} }
       { \Tcircle{\tiny \phantom 1} }\;,
\qquad
\pstree[levelsep=6mm]
       { \Tcircle{\tiny 1} }
       { \Tcircle{\tiny \phantom 1} }
\pstree[levelsep=6mm]
       { \Tcircle{\tiny 2} }
       { \Tcircle{\tiny 3} }\;,
\qquad
\pstree[levelsep=6mm]
       { \Tcircle{\tiny 1} }
       { \pstree[levelsep=5mm,treesep=1mm]
            { \Tcircle{\tiny 2} }
            {{ \Tcircle{\tiny\phantom 1} }{ \Tcircle{\tiny\phantom 1} }}
       }
\hspace{-4mm}
\pstree[levelsep=6mm]
       { \Tcircle{\tiny 3} }
       { \Tcircle{\tiny \phantom 1} }\;,
\qquad
\pstree[levelsep=6mm]
       { \Tcircle{\tiny 1} }
       { \pstree[levelsep=5mm,treesep=1mm]
            { \Tcircle{\tiny 3} }
            {{ \Tcircle{\tiny\phantom 1} }{ \Tcircle{\tiny\phantom 1} }}
       }
\hspace{-4mm}
\pstree[levelsep=6mm]
       { \Tcircle{\tiny 2} }
       { \Tcircle{\tiny \phantom 1} }\;,

\bigskip

\noindent
\pstree[levelsep=6mm]
       { \Tcircle{\tiny 1} }
       { \Tcircle{\tiny \phantom 1} }
\hspace{-4mm}
\pstree[levelsep=6mm]
       { \Tcircle{\tiny 2} }
       { \pstree[levelsep=5mm,treesep=1mm]
            { \Tcircle{\tiny 3} }
            {{ \Tcircle{\tiny\phantom 1} }{ \Tcircle{\tiny\phantom 1} }}
       },
\qquad
\pstree[levelsep=6mm]
       { \Tcircle{\tiny 1} }
       { \pstree[levelsep=5mm,treesep=1mm]
            { \Tcircle{\tiny 2} }
            {{ \Tcircle{\tiny 3} }{ \Tcircle{\tiny\phantom 1} }}
       },
\qquad
\pstree[levelsep=6mm]
       { \Tcircle{\tiny 1} }
       { \pstree[levelsep=5mm,treesep=1mm]
            { \Tcircle{\tiny 2} }
            {{ \Tcircle{\tiny\phantom 1} }{ \Tcircle{\tiny 3} }}
       },
\qquad
\pstree[levelsep=6mm]
       { \Tcircle{\tiny 1} }
       { \pstree[levelsep=5mm,treesep=1mm]
            { \Tcircle{\tiny 2} }
            {{  \pstree[levelsep=5mm,treesep=1mm]
            { \Tcircle{\tiny 3} }
            {{ \Tcircle{\tiny\phantom 1} }{ \Tcircle{\tiny\phantom 1} }} }
            { \Tcircle{\tiny\phantom 1} } }
       },
\qquad
\pstree[levelsep=6mm]
       { \Tcircle{\tiny 1} }
       { \pstree[levelsep=5mm,treesep=1mm]
            {   \Tcircle{\tiny 2} }
            {{ \Tcircle{\tiny\phantom 1} }{  \pstree[levelsep=5mm,treesep=1mm]
            { \Tcircle{\tiny 3} }
            {{ \Tcircle{\tiny\phantom 1} }{ \Tcircle{\tiny\phantom 1}}}}}
       }.

\begin{thm} \label{thtree2}
We have:
\begin{equation}
P_n(t) = \sum_{T\in\mathcal{U}_{n+1}} t^{\emp(T)}, \quad
Q^{(a)}_n(t) = \sum_{T\in\mathcal{F}_n} a^{\cc(T)}t^{\emp(T)}.
\end{equation}
\end{thm}

For example, the forests in $\mathcal{F}_3$ given above illustrate $Q_3=6t^3+5t$.
The last four elements of the list are the trees, and illustrate $P_2=2t^3+2t$.
By counting with a weight 2 on each connected component, we obtain from this 
list that $R_3=24t^3+16t$.

\begin{proof}
From the definition of derivative polynomials in terms of $D$ and $U$, and the
relation $DU-UD=I$, we have:
\begin{equation}
P_n(t) = (D+UUD)^nt, \qquad Q^{(a)}_n(t)=(D+aU+UUD)^n1.
\end{equation}
Let us consider the case of $Q^{(a)}_n$. Let $f_n(D,U)=(D+aU+UUD)^n$ and $c_{n,i,j}$ the coefficient of $U^iD^j$ in its
normal form, as in \eqref{normalorder}. Then from \cite[Theorem 1]{Bla}, $c_{n,i,j}$ counts {\it labelled diagrams} obtained
by connecting three kinds of ``gates'', one for each term in $D+aU+UUD$. More precisely to each term $U^kD^\ell$ we
associate a gate consisting of one node with $k$ outgoing strands and $\ell$ ingoing strands. See the left part of
Figure~\ref{treegate}. Then, $c_{n,i,j}$ count labelled diagrams obtained by connecting $n$ of these gates such that:
\begin{itemize}
 \item $i$ outgoing strands and $j$ ingoing strands are not connected,
 \item all other strands are connected, so that each ingoing strand is connected with an outgoing strand and these form a directed edge,
 \item the gates labelled by the integers in $\{1\dots n\}$, and labels are increasing when we follow a directed edge,
 \item at each node, the ingoing strands on one side and the outgoing strands on another side are ordered,
 \item there is a weight $a$ at each gate corresponding to the term $U$. 
\end{itemize}
We have
\begin{equation}
 Q_n^{(a)}(t) = ( D+aU+UUD  )^n1 = \Big( \sum_{i,j\geq0} c_{i,j} U^iD^j  \Big)1 = \sum_{i\geq0} c_{i,0}t^i
\end{equation}
so that we can only consider diagrams with no unconnected ingoing strand, and $t$ counts the unconnected outgoing strands.

The labelled diagrams described by the above rules are essentially the same as elements in $\mathcal{F}_n$:
it suffices to add an empty leaf at each unconnected outgoing strand to see the equivalence. It is clear that the node
corresponding to the term $U$ will appear exactly once in each connected component of the labelled diagrams, so
the parameter $a$ counts indeed the connected components.

 In the case of $P_n$, we can also consider $f_n(D,U)=(D+UUD)^nU$ and $c_{n,i,j}$ the coefficient of $U^iD^j$ in its
normal form, as in \eqref{normalorder}.
This case is somewhat different since $f_n(D,U)$ is not the $n$th power of some expression, but similar arguments
apply as well: the labelled diagrams that appear have $n+1$ gates, the gate labelled $1$ is of type $U$, all other gates are
of type $D$ or $UUD$. These labelled diagrams are the same as elements in $\mathcal{U}_{n+1}$, as in the previous
case we just have to add an empty leaf to each unconnected outgoing strand to see the equivalence.
As previously said, we refer to \cite{Bla} for precisions about this proof.
\end{proof}

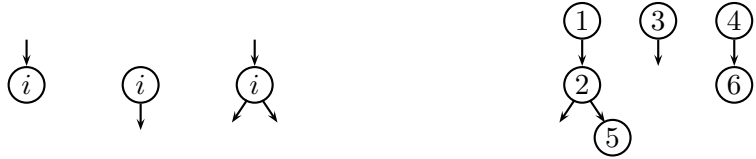
\begin{figure}[h!tp] \psset{unit=5mm}
 \begin{pspicture}(-0.5,0)(6.5,4)  
  \pscircle(0,2){0.5}\rput(0,2){$i$} \psline{->}(0,3.2)(0,2.5)
  \pscircle(3,2){0.5}\rput(3,2){$i$} \psline{->}(3,1.5)(3,0.8)
  \pscircle(6,2){0.5}\rput(6,2){$i$} \psline{->}(6,3.2)(6,2.5)\psline{->}(6.2,1.6)(6.6,1)\psline{->}(5.8,1.6)(5.4,1)
 \end{pspicture}
\hspace{3.5cm}
 \begin{pspicture}(0.4,0)(5.6,4.1) 
  \pscircle(1,2){0.5}\rput(1,2){$2$} \psline{->}(1,3.2)(1,2.5)\psline{->}(1.2,1.6)(1.6,1)\psline{->}(0.8,1.6)(0.4,1)
  \pscircle(1.8,0.6){0.5}\rput(1.8,0.6){$5$}\pscircle(1,3.7){0.5}\rput(1,3.7){$1$}
  \pscircle(3,3.7){0.5}\rput(3,3.7){$3$} \psline{->}(3,3.2)(3,2.5)
  \pscircle(5,3.7){0.5}\rput(5,3.7){$4$} \psline{->}(5,3.2)(5,2.5) \pscircle(5,2){0.5}\rput(5,2){$6$}
 \end{pspicture}
\caption{ The ``gates'' corresponding to terms in $D+aU+DUU$, and a labelled diagram obtained by connecting them. \label{treegate}}
\end{figure}

There is a simple bijection between $\mathcal{T}_{n}$ and $\mathcal{U}_{n+1}$: given $T\in\mathcal{T}_{n}$,
relabel the nodes by $i\mapsto i+1$, then add a new node with label $1$ on top of the root.
There is also a simple bijection between $\mathcal{T}^*_n$ and $\mathcal{F}_n$: let $T\in\mathcal{T}^*_n$,
remove the rightmost leaf, as well as all edges in the path from the root to the rightmost leaf, then the remaining
components form the desired forest. See Figure~\ref{bij_treef} for an example.
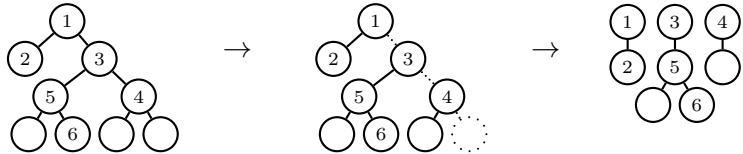
\begin{figure}[h!tp] \psset{unit=4mm}
\pstree[levelsep=5mm,treesep=5mm]
       { \Tcircle{\tiny 1} }
       { 
         { \Tcircle{\tiny 2} }
         { \pstree[levelsep=5mm,treesep=1mm]
            { \Tcircle{\tiny 3} }
            {
              { \pstree[levelsep=5mm,treesep=1mm]{\Tcircle{\tiny 5}}
                {{ \Tcircle{\tiny\phantom 1} }{ \Tcircle{\tiny 6} }}
              }
              { \pstree[levelsep=5mm,treesep=1mm]{\Tcircle{\tiny 4}}
                {{ \Tcircle{\tiny\phantom 1} }{ \Tcircle{\tiny\phantom 1} }}
              }
            }
         }
       }
\quad \begin{pspicture}(0,0) \rput(0,-1){$\to$}\end{pspicture} \qquad
\def\dedge{\ncline[linestyle=dotted,dotsep=0.5mm]}
\pstree[levelsep=5mm,treesep=5mm]
       { \Tcircle{\tiny 1} }
       { 
         { \Tcircle{\tiny 2} }
         { \pstree[levelsep=5mm,treesep=1mm]
            { \Tcircle[edge=\dedge]{\tiny 3} }
            {
              { \pstree[levelsep=5mm,treesep=1mm]{\Tcircle{\tiny 5}}
                {{ \Tcircle{\tiny\phantom 1} }{ \Tcircle{\tiny 6} }}
              }
              { \pstree[levelsep=5mm,treesep=1mm]{\Tcircle[edge=\dedge]{\tiny 4}}
                {{ \Tcircle{\tiny\phantom 1} }{ \Tcircle[linestyle=dotted,edge=\dedge]{\tiny\phantom 1} }}
              }
            }
         }
       }
\quad \begin{pspicture}(0,0) \rput(0,-1){$\to$}\end{pspicture} \qquad
\begin{pspicture}(0,0)(4,0) \rput(2,-1.4){
\pstree[levelsep=6mm]
       { \Tcircle{\tiny 1} }
       { \Tcircle{\tiny 2} }
\hspace{-4mm}
\pstree[levelsep=6mm]
       { \Tcircle{\tiny 3} }
       { \pstree[levelsep=5mm,treesep=1mm]
            { \Tcircle{\tiny 5} }
            {{ \Tcircle{\tiny\phantom 1} }{ \Tcircle{\tiny 6} }}
       }
\hspace{-4mm}
\pstree[levelsep=6mm]
       { \Tcircle{\tiny 4} }
       { \Tcircle{\tiny \phantom 1} }
}\end{pspicture}
\caption{\label{bij_treef}
The bijection from $\mathcal{T}^*_n$ to $\mathcal{F}_n$.
}
\end{figure}

Thus, Theorems \ref{thtree1} and \ref{thtree2} are essentially equivalent although obtained by different methods.
It is also in order to give a bijection between $\mathcal{U}_n$ and $\mathcal{C}^\circ_n$, and by applying this bijection
componentwise it will give a bijection between $\mathcal{F}_n$ and $\mathcal{C}_n$. It is more practical
to give the bijection from $\mathcal{T}_n$ to $\mathcal{S}_n$ (recall that we already have simple bijections
$\mathcal{T}_n\to\mathcal{U}_{n+1}$ and $\mathcal{S}_n\to\mathcal{C}^\circ_{n+1}$), so let $T\in\mathcal{T}_n$.
Consider the ``reading word'' of this tree (it can be defined by $w(T)=w(T_1)\;i\;w(T_2)$ if the tree $T$ has a
root labelled $i$, left son $T_1$ and right son $T_2$). This word contains integers from $1$ to $n$, and
some letters (say $\circ$) to indicate the empty leaves. The first step is to replace each $i$ with $n+1-i$ in
this word. To obtain the snake, replace each integer $i$ by $(-1)^{j+1}i$ where $j$ is the number of $\circ$
before $i$ in the word, then remove all $\circ$. See Figure~\ref{proj} for an example.

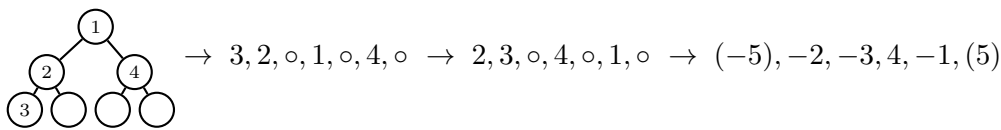
\begin{figure}[h!tp] \psset{unit=4mm}
       { \pstree[levelsep=6mm,treesep=1mm]
            { \Tcircle{\tiny 1} }
            {
              { \pstree[levelsep=5mm,treesep=1mm]{\Tcircle{\tiny 2}}
                {{ \Tcircle{\tiny 3} }{ \Tcircle{\tiny\phantom 1} }}
              }
              { \pstree[levelsep=5mm,treesep=1mm]{\Tcircle{\tiny 4}}
                {{ \Tcircle{\tiny\phantom 1} }{ \Tcircle{\tiny\phantom 1} }}
              }
            } } 
\begin{pspicture}(-0.8,0)(25.5,0) 
\rput(12,-1){$\to\; 3,2,\circ,1,\circ,4,\circ\;\to\; 2,3,\circ,4,\circ,1,\circ\;\to\; (-5),-2,-3,4,-1,(5)$}
\end{pspicture}
\caption{The bijection from $\mathcal{T}_n$ to $\mathcal{S}_n$. \label{proj} }
\end{figure}

Actually, this bijection can be highlighted by Proposition~\ref{propval} and the discussion leading to it.
Indeed, it is just a variant of a classical bijection between permutations and unary-binary increasing trees \cite{Vie},
where double ascents and double descents correspond to nodes having only one child, and valleys correspond
to leaves. Via these bijections, removing the empty leaves of $T\in\mathcal{T}_n$ is the same as taking the 
absolute value of a snake $S\in\mathcal{S}_n$. In the other direction, $T$ seen as a unary-binary tree together 
with the data of the empty leaves, is the same as a permutation together with a choice of signs making it into 
a snake.

\section*{Conclusion}

There is still a wide range of results on alternating permutations that might have a counterpart in the 
case of snakes, for example one could ask if André permutations and simsun permutations \cite{Sta},
both counted by $E_n$, have natural generalizations in signed permutations (as a first answer, see \cite{purtill}
for a definition of {\it André signed permutation}, counted by $S_n$). Another problem
is to give more combinatorial meaning to the Arnold-Seidel triangles \cite{Arn,Dum} used to compute the
integers $S_n$ (see \cite{GZ} for recent work on this subject in the case of integers $E_n$).
As for $q$-analogs, we have generalized the numbers $E_n(q)$ defined in \cite{HZR} but there is another
$q$-analog of $E_n$ related with inversions and major index \cite{Sta} that might also be extended to snakes.

Though we only discussed the type $B$ ones, Arnol'd~\cite{Arn} also defined some type $D$ snakes counted
by the Springer number $K(D_n)$. Besides, Hoffman showed that $K(D_n)$ is equal to $P_n(1)-Q_n(1)$, so that
it is also the number of $\pi\in\mathcal{S}_n$ such that $\pi_1<0$. As usual with type $D$, it is much
more difficult to obtain enumerative or bijective results with either of these two families, they are not even a
subset of the group of even-signed permutation $D_n$ as one could expect. Having said that, one possible
direction towards a better understanding of these objects could be the geometric definition of the Springer
number $K(W)$ in terms of hyperplane arrangements \cite{Arn}. For example, in this context it would be very
interesting to give the polynomial $Q_n(t)$ a geometric meaning that would refine
$Q_n(1)=K(\mathfrak{S}^B_n)$.

\section*{Acknowledgement}

Part of this research was done during a visit of the LABRI in Bordeaux, and I thank
all the Bordelais for welcoming me and for various suggestions concerning this work.

\bigskip


\bigskip

\end{document}